\documentclass[12pt,oneside,a4paper]{amsart}

\usepackage{amssymb}
\usepackage{amsmath}
\usepackage{amsthm}
\usepackage{amscd}
\usepackage[all]{xy}
\usepackage{longtable}
\usepackage{mathrsfs}
\usepackage[dvipdfmx]{xcolor}
\usepackage[dvipdfmx]{pict2e}
\usepackage[dvipdfmx]{graphicx}

\usepackage{comment}
\usepackage{todonotes}

\usepackage[dvipdfmx]{hyperref}
\usepackage{hyperref}
\hypersetup{
	colorlinks=true,
	linkcolor=red,
	citecolor=blue}
	\usepackage[utf8]{inputenc}
\usepackage[T1]{fontenc}
\usepackage{color}

\usepackage{extarrows}
\usepackage{tikz}
\usetikzlibrary{cd}

\usepackage{tikz}
\usepackage{tikz-cd}
\usetikzlibrary{arrows.meta}

\usepackage{amscd}

\usepackage{geometry}
\theoremstyle{plain}
\newtheorem{theorem}{Theorem}[section]
\newtheorem{lemma}[theorem]{Lemma}
\newtheorem{corollary}[theorem]{Corollary}
\newtheorem{proposition}[theorem]{Proposition}
\newtheorem{remark}[theorem]{Remark}

\theoremstyle{definition}
\newtheorem{definition}[theorem]{Definition}
\newtheorem{example}[theorem]{Example}
\newtheorem{conjecture}[theorem]{Conjecture}

\newtheorem{remark-theorem}[theorem]{Remark-Theorem}
\newtheorem*{observation}{Observation}

\newcommand{\kah}{K\"{a}hler }
\newcommand{\idd}{i\partial\overline{\partial}}
\newcommand{\dbar}{\overline{\partial}}

\newcommand{\cal}[1]{\mathcal{#1}}
\newcommand{\bb}[1]{\mathbb{#1}}
\newcommand{\scr}[1]{\mathscr{#1}}
\newcommand{\rom}[1]{\mathrm{#1}}
\newcommand{\fra}[1]{\mathfrak{#1}}

\newcommand{\xs}{X_{sing}}
\newcommand{\reg}{X_{reg}}
\newcommand{\ogr}{\omega_X^{GR}}
\newcommand{\tx}{\widetilde{X}}

\newcommand{\tl}[1]{\widetilde{#1}}

\newcommand{\iO}[1]{i\Theta_{#1}}


\begin{document}
\title
[Global Embeddings and Approximation theorems] 
{Global Embeddings \\ of Weakly Pseudoconvex Complex Spaces \\ and Refined Runge-type Approximation Theorems}
\author{Yuta Watanabe}
\address{Department of Mathematics, Faculty of Science and Engineering, Chuo University.
1-13-27 Kasuga, Bunkyo-ku, Tokyo 112-8551, Japan}
\email{{\tt wyuta.math@gmail.com}, {\tt wyuta@math.chuo-u.ac.jp}}

\begin{abstract}
    Runge-type approximation principles for holomorphic sections of adjoint line bundles are known only for weakly pseudoconvex manifolds. 
    In this paper, we establish a refined form of such principles adapted to the setting of complex spaces and show that they yield global holomorphically embeddings for the regular locus of weakly pseudoconvex complex spaces.
    The key point is the construction of sequences of singular Hermitian metrics after a canonical resolution of singularities, together with a control of multiplier ideal sheaves via the strong openness property. 
    This refined Runge-type approximation principle enables the globalization of local sections even when singularities persist at infinity. As an application, we solve the Union problem for weakly pseudoconvex complex manifolds.
\end{abstract}


\maketitle



\vspace{-7mm}

\section{Introduction}

In complex geometry, the notions of positivity and embedding theorems have played a very important role. 
Indeed, when a complex manifold is compact, Kodaira's embedding theorem \cite{Kod54} implies that being projective is characterized by the existence of a Hodge metric, 
and, with respect to positivity, the notions of \textit{ample} and \textit{positive} coincide.
After that, Kodaira's result was extended to compact complex spaces by Grauert \cite{Gra62}, and to (non-compact) weakly pseudoconvex manifolds 
by Takayama \cite{Tak98}. 
Weakly pseudoconvexity includes both compact and Stein cases and forms the broadest and highly significant class among the complex-geometric objects that can be treated.
In fact, every complex Lie group is always weakly pseudoconvex (see \cite{Kaz73}).
Here, a function $\varPsi:X\longrightarrow[-\infty,+\infty)$ on a complex space $X$ is \textit{exhaustion} if all sublevel sets $X_c:=\{x\in X\mid\varPsi(x)<c\}$, $\forall\,c\in\bb{R}$, are relatively compact. 
A complex space is said to be \textit{weakly} \textit{pseudoconvex} if there exists a smooth exhaustion plurisubharmonic function.

In this paper, embeddings of weakly pseudoconvex complex spaces equipped with a positive line bundle are studied.
In this setting, for each sublevel set $X_c$ with singularities, embeddings exist due to relative compactness, and it is known that ampleness coincides with positivity on $X_c$ \cite{Fuj75}. 
However, globally, non-compactness prevents the positive line bundle $L\longrightarrow X$ from being necessarily ample; indeed, Ohsawa provided a counterexample where this fails \cite{Ohs79}. 
In general, when there are no singularities, as shown by Takayama \cite[Theorem\,1.2]{Tak98}, 
for any integer $m\geq n(n+1)/2$, the adjoint bundle $K_X\otimes L^{\otimes m}$ becomes ample and yields a global embedding, where $n=\rom{dim}\,X$. 

Recently, it has been shown that, for a weakly pseudoconvex manifold $X$ equipped with a holomorphic line bundle $L \longrightarrow X$ carrying a singular positive Hermitian metric $h$, 
global embeddings (see \cite[Theorem\,1.2]{Wat24}) can be obtained via the adjoint bundle $K_X\otimes L^{\otimes m}$ while avoiding the essential singular locus of $h$, $E_{+}(h) := \bigcup_{k \in \mathbb{N}} V(\scr{I}(h^k))=\{x\in X\mid\nu(-\log h,x)>0\}$ (see \cite[Proposition\,3.11]{Wat24}). 
Furthermore, the adjoint bundle is big (see \cite[Theorem\,1.7]{Wat24}). 
However, the integer $m$ cannot, in general, be taken optimally as $m \geq n(n+1)/2$ as in Takayama's result; rather, there exist a constant $C\geq1$ and a smooth Hermitian metric $h_0$ satisfying the necessary conditions such that $m \geq C\cdot n(n+1)/2$ is required. 
We present the following main results.

\begin{theorem}\label{Theorem: embedding theorem}
    Let $X$ be a non-compact weakly pseudoconvex complex space of pure dimension $n$ 
    and $L\longrightarrow X$ be a holomorphic line bundle.
    If $L$ is positive, then the regular locus $\reg$ of $X$ can be holomorphically embedded into $\bb{P}^{2n+1}$. 
    
    Furthermore, let $\pi:\tx\longrightarrow X$ be a canonical desingularization as in Theorem \ref{Theorem: canonical desingularization} and $E:=\pi^{-1}(\xs)$ be a $\pi$-exceptional divisor, 
    then the adjoint bundle $K_{\tx}\otimes\pi^*L^{\otimes m}$ is ample over $\tx\setminus E$ for any $m\geq n(n+1)/2$. 
    In fact, $\tx\setminus E\cong\reg$ is holomorphically embeddable into $\bb{P}^{2n+1}$ by a linear subsystem of $|(K_{\tx}\otimes \pi^*L^{\otimes m})^{\otimes (n+2)}|$ for $m\geq n(n+1)/2$.
\end{theorem}

\begin{theorem}\label{Theorem: bigness theorem and Vol>0}
    Let $X$ be a non-compact weakly pseudoconvex complex space of pure dimension $n$ and $\pi:\tx\longrightarrow X$ be a canonical desingularization as in Theorem \ref{Theorem: canonical desingularization}. 
    Let $L\longrightarrow X$ be a holomorphic line bundle. If $L$ is positive, then the adjoint bundle $K_{\tx}\otimes\pi^*L^{\otimes m}$ is big for any $m\geq n(n+1)/2$. 

    In particular, if $m\geq n(n-1)/2+1$ then the adjoint bundle $K_{\tx}\otimes\pi^*L^{\otimes m}$ is generated by global sections on $\tx\setminus E$.
    Furthermore, if there exists $c>\inf_X\varPsi$ such that $\xs\bigcap X_c=\emptyset$, then we have the inequality $\rom{Vol}_{\tx}(K_{\tx}\otimes\pi^*L^{\otimes m+1})>0$, where $m\geq n(n-1)/2+1$.
\end{theorem}

Throughout this paper, \textit{canonical desingularization} is understood in the sense of Theorem \ref{Theorem: canonical desingularization}, and a complex space is always assumed to be pure dimension, without necessarily being reduced.

These results provide a global extension of the results of Grauert \cite{Gra62} and Fujiki \cite{Fuj75} to weakly pseudoconvex complex spaces, which may be non-compact, and also generalize Takayama's result \cite{Tak98} to complex spaces with singularities.

\begin{observation}
    On each sublevel set $X_c$, we take a resolution of singularities $\pi : \tx_c \longrightarrow X_c$. 
    Then, by using the Negativity Lemma (see Lemma \ref{Lemma: Negativity lemma}) and the relative compactness of $X_c$, we can construct on $\tx_c$ a singular positive Hermitian metric $\tl{h}_c$ on $\pi_c^*L|_{\tx_c}$ that becomes smooth on $\tx_c\setminus E_c$, 
    where $E_c$ denotes the $\pi$-exceptional divisor. Consequently, on each sublevel set $X_c$, it follows from \cite{Wat24} that $\tx_c\setminus E_c\cong X_c\setminus\xs$ admits an embedding via the adjoint bundle $K_{\tx_c}\otimes \pi^*L^{\otimes m}$. 
    Here, Theorem \ref{Theorem: embedding theorem} and \ref{Theorem: bigness theorem and Vol>0} correspond to the case $C = 1$ in \cite[Theorem\,1.2]{Wat24}, which is optimal.
\end{observation}

However, globally, it is unclear whether the Negativity Lemma can be applied, and it is unknown whether we can construct a singular \textit{positive} Hermitian metric on $\pi^*L\longrightarrow\tx$ that is defined \textit{globally} on $\tx$. 
Therefore, the result of \cite[Theorem\,1.2]{Wat24} cannot be applied directly. 
The uncertainty regarding the existence of a globally defined singular positive Hermitian metric arises as the \textit{main difficulty} in obtaining a global embedding. 
Here, Theorem \ref{Theorem: bigness theorem and Vol>0} provides an example in which the adjoint bundle becomes big even though no globally defined singular positive Hermitian metric exists on $\pi^*L\longrightarrow\tx$.

We introduce the following definition and provide a refined Runge-type approximation Theorem \ref{Theorem: ext approximation theorem} for holomorphic sections to overcome this difficulty, using which we obtain a global embedding theorem \ref{Theorem: embedding theorem}. 

\begin{definition}\label{Definition: exhaustion positive}
    Let $(X,\varPsi)$ be a weakly pseudoconvex manifold and $\scr{J}$ be an ideal sheaf on $X$. 
    We say that a holomorphic line bundle $L\longrightarrow X$ is 
    \textit{exhaustion} \textit{singular}-\textit{positive} \textit{with} $\scr{J}$, 
    if there exist an increasing sequence of real numbers $\{c_j\}_{j\in\mathbb{N}}$ diverging to $+\infty$ and a sequence of singular Hermitian metrics $\{h_j\}_{j\in\bb{N}}$ such that each $h_j$ is defined on $L|_{X_{c_j}}$, is singular positive on $X_{c_j}$, and satisfies $\scr{I}(h_j)=\scr{J}$ on $X_{c_j}$.

    In particular, if each metric $h_j$ in the sequence $\{h_j\}_{j\in\bb{N}}$ is smooth, we simply say that $L$ is \textit{exhaustion positive}. 
\end{definition}

Here, the Runge-type approximation theorem regarding positive line bundles is established by Nakano, Kazama and Ohsawa, and is one of most important conclusion of the cohomology theory on weakly pseudoconvex manifolds (see \cite{Kaz73b,Ohs82,Ohs83}), 
and it plays a central role in globalizing embeddings (see \cite{Tak98}).
In the following, we present a refined Runge-type approximation regarding exhaustion singular positivity. 

\begin{theorem}\label{Theorem: ext approximation theorem}
    Let $X$ be a weakly pseudoconvex manifold, $L\longrightarrow X$ be a holomorphic line bundle and $\scr{J}$ be an ideal sheaf on $X$. 
    If $L$ is exhaustion singular-positive with $\scr{J}$, 
    then for any sublevel set $X_c$, the natural restriction map
    \begin{align*}
        \rho_c:H^0(X,K_X\otimes L)\longrightarrow H^0(X_c,K_X\otimes L\otimes\scr{J})
    \end{align*}
    has dense images with respect to 
    the topology of uniform convergence on all compact subsets in $X_c$.
\end{theorem}

Furthermore, we obtain a full generalization of the Runge-type approximation theorem (see \cite[Theorem 1]{Kaz73b}) to complex spaces. 

\begin{theorem}\label{Theorem: ext approximation theorem to complex spaces}
    Let $X$ be a reduced weakly pseudoconvex complex space of pure dimension $n$ and $L\longrightarrow X$ be a holomorphic line bundle. 
    If $L$ is positive, then for any sublevel set $X_c$, the natural restriction map
    \begin{align*}
        \rho_c^\pi:H^0(X,\ogr\otimes\cal{O}_X(L))\longrightarrow H^0(X_c,\ogr\otimes\cal{O}_X(L)),
    \end{align*}
    induced by the canonical desingularization $\pi:\tx\longrightarrow X$ and the natural restriction map $\rho_{\tx,c}$ on $\tx$ associated with the adjoint bundle $K_{\tx}\otimes\pi^*L$,
    has dense images with respect to the topology of uniform convergence on all compact subsets in $X_c$.
\end{theorem}

Here, the sheaf $\ogr$ denotes the \textit{Grauert}-\textit{Riemenschneider canonical sheaf} on $X$. 
As a direct application of the refined approximation Theorem \ref{Theorem: ext approximation theorem}, an embedding is obtained assuming only the existence of an exhaustion positive line bundle (see Theorem \ref{Theorem: embedding only exhaustion positive}). 
In particular, for a toroidal group $T$, if there exists a holomorphic line bundle $L\longrightarrow T$ with a \kah form in the first Chern class $c_1(L)$, then $T$ can be holomorphically embedding into $\mathbb{P}^{\,2\dim T+1}$ (see Corollary \ref{Corollary: embedding of toroidal group}).
Such line bundles on toroidal groups (see Example \ref{Example: Tak98b Theorem}), as well as line bundles on the product of a Stein manifold $S$ and a compact manifold $X$ that are positive on each compact fiber $\{p\} \times X$ (see Examples \ref{Example: Tak98b Lemma} and \ref{Example: Stein times weakly pseudoconvex}), provide natural examples of exhaustion positive line bundles.
As an important and typical example of exhaustion singular-positivity, in the setting of Theorem \ref{Theorem: embedding theorem}, 
we show in $\S\ref{Section: exhaustion positivity}$ that $\pi^*L\longrightarrow\tx$ is exhaustion singular-positive with $\cal{O}_{\tx}$ (see Theorem \ref{Theorem: exh sing positivity on tlX}).
In particular, in order to obtain a global embedding, it is also necessary that, for a holomorphic line bundle $P\longrightarrow\tx$ equipped with a singular semi-positive Hermitian metric $h_P$, 
the twisted line bundle $\pi^*L\otimes P$ is exhaustion singular-positive with $\scr{I}(h_P)$ (see Theorem \ref{Theorem: exh sing positivity on tlX twisted by psef}), which can be proved using the strong openness property \cite{GZ15}.

The following question has held great interest in Stein theory: 

\vspace{3mm}
\noindent
\textbf{The Union Problem.}  \textit{Is a reduced complex space $\Omega=\bigcup_{\nu\in\bb{N}}\Omega_\nu$, given as an increasing union $\Omega_1\subset \Omega_2\subset \Omega_3\subset\cdots$ of open Stein subspaces, itself a Stein space?} 
\vspace{3mm}

Every compact analytic subspace of such a space $\Omega$ is certainly finite.
This problem was raised by Behnke-Thullen \cite{BT33} in the case where $\Omega$ is an open subset of $\bb{C}^n$, and it was solved for various special domains. 
After that, it was solved affirmatively by Behnke-Stein \cite{BS39}.
Stein provided a sufficient condition for the Union Problem:

\begin{theorem}[\textnormal{\cite[Satz 1.3]{Ste56}}]\label{Theorem: Stein's result}
    Let $\Omega$ be a reduced complex space and $\Omega_1\subset \Omega_2\subset\cdots$ be an increasing sequence of Stein domains with $\bigcup_{\nu\in\bb{N}}\Omega_\nu=\Omega$. 
    If every pair $(\Omega_\nu,\Omega_{\nu+1})$ is Runge, i.e., $\cal{O}(\Omega_{\nu+1})$ dense in $\cal{O}(\Omega_\nu)$ in the topology of compact convergence, then $\Omega=\bigcup_{\nu\in\bb{N}}\Omega_\nu$ is Stein.
\end{theorem}

Nevertheless, if a Runge condition is not imposed on the $\Omega_\nu$'s, the space $X$ is not necessarily holomorphically convex.
In fact, Forn{\ae}ss constructed in \cite{For76,For77} a non-holomorphically convex manifold as a \textit{counterexample} to the Union Problem. 
In other words, the Union Problem \textit{cannot} be solved in general. 
Markoe and Silve showed in \cite{Mar77,Sil78} that the Union Problem is solvable if and only if $H^1(\Omega,\mathcal{O}_{\Omega}) = 0$.

Here, we present the following as an application of approximation theorems.

\begin{theorem}\label{Theorem: Union Problem on w.p.c}
    The Union Problem can be solved when the union is a weakly pseudoconvex manifold. 
    In other words, if a weakly pseudoconvex manifold $\Omega$ is given as an increasing union $\Omega=\bigcup_{\nu\in\bb{N}}\Omega_\nu$ of open Stein subspaces $\Omega_1\subset\Omega_2\subset\cdots$, then $\Omega$ is Stein. 
\end{theorem}

Forn{\ae}ss and Narasimhan posed the following problem in \cite{FN80}: if $\Omega=\bigcup_{\nu\in\bb{N}}\Omega_\nu$ is an open subset of a Stein space, does the Union Problem admit a solution? 
This remains an important open problem on Stein spaces with singularities, and at present, the term \textit{Union Problem} often refers to this formulation.

\section{Preliminaries}

\subsection{Canonical desingularization of complex spaces}

Even in the non-compact case, it is known that a global resolution of singularities can be obtained, which is locally given by a finite sequence of blow-ups with smooth centers. 
This is achieved by patching together the resolutions 
constructed on relatively compact subsets.

\begin{theorem}[\textnormal{\cite[Theorem\,13.3 and 13.4]{BM97}}]\label{Theorem: canonical desingularization}
    Let $X$ be a complex space which is not necessarily compact or reduced. 
    There exists a desingularization $\pi:\tx\longrightarrow X$, which is a composite of a locally finite sequence of blow-ups, such that 
    \begin{itemize}
        \item the map $\pi:\tx\longrightarrow X$ is proper holomorphic.
        \item the set $\tx$ is smooth, 
        and the $\pi$-exceptional set $E:=\pi^{-1}(\xs)$ is simple normal crossing, where $E$ denotes the collection of all exceptional divisors.
        \item for any relatively compact open subset $V$ of $X$, the restriction $\pi|_V:\tx|_{\pi^{-1}(V)}\longrightarrow X|_V$ is a composite of a finite sequence of blow-ups with smooth centres. 
        \item the restriction $\pi|_{\tx\setminus E}:\tx\setminus E\longrightarrow X\setminus\xs=\reg$ is biholomorphic.
        \item $\pi$ is canonical in the sense that for any isomorphic $\varphi:X|_U\overset{\cong}{\longrightarrow} X|_V$, where $U$ and $V$ are open subsets of $X$, lifts to an isomorphism $\widetilde{\varphi}:\tx|_{\pi^{-1}(U)}\longrightarrow\tx|_{\pi^{-1}(V)}$.
    \end{itemize}
\end{theorem}

The desingularization in this theorem is referred to as the \textit{canonical} \textit{desingularization}.

\begin{lemma}[{Negativity Lemma, cf. \cite[Remark 1.6.2 (2)]{Kaw24}, the proof of \cite[Lemma\,3.6]{Wat24}}]\label{Lemma: Negativity lemma}
    Let $X$ be a complex space, 
    $\pi:\tx\longrightarrow X$ be a canonical desingularization and $E=\sum_{j\in\cal{J}}E_j:=\pi^{-1}(\xs)$ be the $\pi$-exceptional set which is a reduced simple normal crossing divisor. 
    Then, for any relatively compact open subset $V$ of $X$, there exists an effective divisor $E_V:=\sum_{j\in J_V}b_jE_j$ with $b_j\in\bb{N}$ such that the line bundle $\cal{O}_{\widetilde{X}}(-E_V)$ admits a smooth Hermitian metric $h^*_{E_V}$ on $\widetilde{V}:=\pi^{-1}(V)$ 
    whose curvature compensates for the loss of positivity of the pullback $\pi^*\omega$ on $\tl{V}$, where $\pi^*\omega$ degenerates along $E$ for any Hermitian metric $\omega$ on $X$ and $J_V:=\{j\in\cal{J}\mid\pi(E_j)\cap V\ne\emptyset\}$.  
    This metric $h^*_{E_V}$ induced from the normal bundles of the components of the divisor, and for any Hermitian metric $\omega$ on $X$, there exists $\varepsilon_V>0$ such that 
    \begin{align*}
        \pi^*\omega+\varepsilon\iO{\cal{O}_{\widetilde{X}}(-E_V),h^*_{E_V}}>0
    \end{align*}
    on $\tl{V}$ for any $0<\varepsilon<\varepsilon_V$.
    In particular, the line bundle $\cal{O}_{\widetilde{X}}(-E_V)|_{\tl{V}}$ is $\pi$-ample. 
\end{lemma}

\subsection{Plurisubharmonicity and positivity on complex spaces}

Let $X$ be a complex space of pure dimension $n$. 

\begin{definition}[\textnormal{\cite[Chapter\,V, Definition\,1.4]{GPR94}, \cite{Fuj75}}]
    A function $\varphi:X\longrightarrow[-\infty,+\infty)$ is called (resp. \textit{strictly}) \textit{plurisubharmonic} if for any $x\in X$ 
    there exist an open neighborhood $U$ admitting a closed embedding $\iota_U:U\hookrightarrow V\subset\bb{C}^N$, here $V$ is an open subset of $\bb{C}^N$, 
    and a (resp. strictly) plurisubharmonic function $\widetilde{\varphi}:V\longrightarrow[-\infty,+\infty)$ such that $\varphi|_U=\widetilde{\varphi}\circ\iota_U$.
\end{definition}

Concerning the smoothness of functions, it is defined similarly by requiring that it can be expressed as the restriction of smooth functions on $V$.
Moreover, a function is said to be \textit{quasi}-\textit{plurisubharmonic} if it can be written locally as the sum of a smooth function and a plurisubharmonic function.

Let $L\longrightarrow X$ be a holomorphic line bundle. Then there exist an open covering $\{U_\alpha\}_{\alpha\in\Lambda}$ of $X$ and isomorphisms $\iota_\alpha:L|_{U_\alpha}\overset{\cong}{\longrightarrow}U_\alpha\times\bb{C}$. 
Conversely, the set of holomorphic functions $\{f_{\alpha\beta}\}_{\alpha,\beta\in\Lambda}$, where $f_{\alpha\beta}:=\iota_\alpha\circ\iota_\beta^{-1}|_{U_\alpha\cap U_\beta}\in \varGamma(U_\alpha\cap U_\beta,\cal{O}_X)$, defines $L$. 
Such a $(\{f_{\alpha\beta}\},\{U_\alpha\})_{\Lambda}$ is called a \textit{system of transition functions of} $L$.

\begin{definition}
    We say that $h$ is \textit{smooth Hermitian metric} on $L$ if for a system of transition functions $(\{f_{\alpha\beta}\},\{U_\alpha\})_{\Lambda}$ of $L$, there exists a collection of positive smooth functions $h=\{h_\alpha:U_\alpha\longrightarrow\bb{R}_{>0}\}_{\alpha\in\Lambda}$ such that $h_\beta/h_\alpha=|f_{\alpha\beta}|^2$ on $U_\alpha\cap U_\beta$.
\end{definition}

For any trivialization $\tau:L|_U\overset{\cong}{\longrightarrow}U\times\bb{C}$, a smooth Hermitian metric $h$ on $L$ can be expressed as 
\begin{align*}
    ||\xi||_h=|\tau(\xi)|e^{-\varphi(x)}, \quad x\in U, \,\xi\in L_x
\end{align*}
using a smooth function $\varphi$ on $U$. 
The smooth function $\varphi:U\longrightarrow\bb{R}$ is called the \textit{weight function of} $h$ \textit{with respect to the trivialization} $\tau$.

For a smooth Hermitian metric $h=\{h_\alpha:U_\alpha\longrightarrow\bb{R}_{>0}\}_{\alpha\in\Lambda}$ with respect to a system of transition functions $(\{f_{\alpha\beta}\},\{U_\alpha\})_{\Lambda}$, 
the weight function of $h$ with respect to any trivialization $\tau:L|_U\overset{\cong}{\longrightarrow}U\times\bb{C}$ can be expressed on $U\cap U_\alpha$ as $\varphi=-\frac{1}{2}\log h_\tau=-\frac{1}{2}\log h_\alpha+\log|\iota_\alpha\circ\tau^{-1}|$, 
here the new element $h_\tau$ of the collection $h=\{h_\alpha\}_{\alpha\in\Lambda}$ is defined by $h_\tau=h_\alpha|\iota_\alpha\circ\tau^{-1}|^2$ on $U\cap U_\alpha$. 
In particular, if $\tau = \iota_\alpha$, then $\varphi = -\frac{1}{2}\log h_\alpha$.

\begin{definition}
    A smooth Hermitian metric $h$ on $L$ is said to be \textit{positive} if there exists a collection of positive smooth functions $h=\{h_\alpha\}_{\alpha\in\Lambda}$ on $L$ such that $-\log h_\alpha$ is strictly plurisubharmonic on $U_\alpha$ for all $\alpha\in\Lambda$. 
    In other words, the weight function of $h$ with respect to any trivialization $\tau$ is smooth strictly plurisubharmonic. 

    A holomorphic line bundle $L\longrightarrow X$ is said to be \textit{positive} if there exists a smooth Hermitian metric $h$ on $L$ which is positive. 
\end{definition}

\subsection{Singular Hermitian metric and exhaustion positivity}

In this subsection, let $X$ be a complex manifold, and introduce the notions of singular Hermitian metrics on holomorphic line bundles and their positivity.

\begin{definition}[\textnormal{\cite{Dem90,Dem12}}]
    A \textit{singular} \textit{Hermitian} \textit{metric} $h$ on a holomorphic line bundle $L\longrightarrow X$ is a metric for which 
    the weight function $\varphi$ of $h$ with respect to any local trivialization $\tau:L|_U\stackrel{\simeq}{\longrightarrow}U\times\bb{C}$ is locally integrable, i.e., $\varphi\in L^1_{loc}(U)$.
    
    
    In other words, for a fiexd smooth Hermitian metric $h_0$, the singular Hermitian metric $h$ can be written as $h=h_0 e^{-2\varphi}$ for some locally integrable function $\varphi\in L^1_{loc}(X,\bb{R})$. 
\end{definition}

\begin{definition}[\textnormal{\cite[Definition\,3.2]{Wat23}}]
    Let $h$ be a singular Hermitian metric on a holomorphic line bundle $L\longrightarrow X$. We say that a singular Hermitian metric $h$ is 
    \begin{itemize}
        \item \textit{singular} \textit{semi}-\textit{positive} if $\iO{L,h}\geq0$ in the sense of currents, i.e., the weight of $h$ with respect to any trivialization coincides with some plurisubharmonic function almost everywhere. 
        \item \textit{singular} \textit{positive} if $\iO{L,h}\geq\varepsilon\omega$ in the sense of currents for some Hermitian metric $\omega>0$ and some smooth positive function $\varepsilon:X\longrightarrow\bb{R}_{>0}$, 
        i.e., the weight of $h$ with respect to any trivialization coincides with some strictly plurisubharmonic function almost everywhere. 
    \end{itemize}
\end{definition}

Note that, being singular semi-positive is coincides with being pseudo-effective on compact complex manifolds. 
Furthermore, being singular positive also coincide with begin big on compact \kah manifolds by Demailly's characterization (see \cite{Dem90}, \cite[Chapter6]{Dem12}).

\begin{definition}
    A function $\varphi:X\longrightarrow[-\infty,+\infty)$ is said to be a \textit{quasi}-\textit{plurisubharmonic} if $\varphi$ is locally the sum of a plurisubharmonic function and a smooth function. 
    The \textit{multiplier} \textit{ideal} \textit{sheaf} $\scr{I}(\varphi)\subset\cal{O}_X$ is defined by 
    \begin{align*}
        \scr{I}(\varphi)(U):=\{f\in\cal{O}_X(U)\mid |f|^2 e^{-2\varphi}\in L^1_{loc}(U)\}
    \end{align*}
    for any open subset $U\subset X$. For a singular Hermitian metric $h$ on $L$ with the local weight $\varphi$, i.e., $h=h_0e^{-2\varphi}$, we define the multiplier ideal sheaf of $h$ by $\scr{I}(h):=\scr{I}(\varphi)$.
\end{definition}

The $\it{Lelong}$ $\it{number}$ of a quasi-plurisubharmonic function $\varphi$ on $X$ is defined by 
\begin{align*}
    \nu(\varphi,x):=\liminf_{z\to x}\frac{\varphi(z)}{\log|z-x|}
\end{align*}
for some coordinate $(z_1,\cdots,z_n)$ around $x\in X$.
For the relationship between the Lelong number of $\varphi$ and the integrability of $e^{-\varphi}$, the following important result obtained by Skoda in \cite{Sko72} is known; If $\nu(\varphi,x)<1$ then $e^{-2\varphi}$ is integrable around $x$.
From this, particularly if $\nu(-\log h,x)<2$ then $\mathscr{I}(h)=\mathcal{O}_{X,x}$ immediately.

\begin{definition}[\textnormal{=\,Definition\,\ref{Definition: exhaustion positive}}]
    Let $(X,\varPsi)$ be a weakly pseudoconvex manifold and $\scr{J}$ be an ideal sheaf on $X$. 
    We say that a holomorphic line bundle $L\longrightarrow X$ is 
    \textit{exhaustion} \textit{singular}-\textit{positive} \textit{with} $\scr{J}$, 
    if there exist an increasing sequence of real numbers $\{c_j\}_{j\in\mathbb{N}}$ diverging to $+\infty$ and a sequence of singular Hermitian metrics $\{h_j\}_{j\in\bb{N}}$ such that each $h_j$ is defined on $L|_{X_{c_j}}$, is singular positive on $X_{c_j}$, and satisfies $\scr{I}(h_j)=\scr{J}$ on $X_{c_j}$.
    
    In particular, if each metric $h_j$ in the sequence $\{h_j\}_{j\in\bb{N}}$ is smooth, we simply say that $L$ is \textit{exhaustion positive}.
\end{definition}

It is clear that when $X$ is compact, these notions simply coincides with being singular positive (resp. positive).
exhaustion (singular)-positivity is a concept that arises very naturally in the following way.

\begin{example}\label{Example: exhaustion positivity}
    Let $X$ be a weakly pseudoconvex manifold. 
    If $X$ is exhaustable by an increasing sequence of Stein open subsets (for instance, if $X$ itself is Stein), then any holomorphic line bundle $L\longrightarrow X$ is exhaustion positive.  
    Furthermore, for any singular Hermitian metric $h$ on $L$ whose curvature is locally bounded from below in the sense of currents, the line bundle $L$ is exhaustion singular-positive with $\scr{I}(h)$.
\end{example}

\begin{example}[\textnormal{cf.\,\cite[Lemma\,2.3]{Tak98b}}]\label{Example: Tak98b Lemma}
    Let $S$ be a Stein manifold and $X$ be a weakly pseudoconvex manifold. 
    Here, let $Y := S \times X$, then $Y$ is weakly pseudoconvex; in particular, $X$ may be compact.
    Let $L$ be a holomorphic line bundle on $Y=S\times X$.
    Assume that the restriction $L|_{p\times X}$ is positive for every $p\in S$ (regarded as a line bundle on $X$).
    Then $L$ 
    is exhaustion positive.
\end{example}

\begin{example}\label{Example: Stein times weakly pseudoconvex}
    Let $S$ be a Stein manifold, $X$ be a weakly pseudoconvex manifold and $L\longrightarrow X$ be a holomorphic line bundle. 
    Here, let $Y:=S\times X$ and denote by $p_X:Y\longrightarrow X$ the natural projection. 
    If $L$ is exhaustion positive then the holomorphic line bundle $p_X^*L\longrightarrow Y$ is also exhaustion positive.
\end{example}

\begin{example}[\textnormal{cf.\,\cite[Theorem\,3.1]{Tak98b}}]\label{Example: Tak98b Theorem}
    Let $L$ be a holomorphic line bundle on a toroidal group $T$ with a \kah form in the first Chern class $c_1(L)$. 
    Then $L$ 
    is exhaustion positive.
\end{example}

Here, every toroidal group is a complex Lie group and hence weakly pseudoconvex, and Example \ref{Example: Tak98b Theorem} is obtained by weak $\partial\overline{\partial}$-lemma (see \cite[Weak $\partial\overline{\partial}$-Lemma 3.14]{Tak98b}).

\begin{definition}[\textnormal{cf.\,\cite[Definition\,6.3]{Dem12}}]
    Let $L$ be a holomorphic line bundle on a complex manifold. 
    Consider two singular Hermitian metrics $h_1$, $h_2$ on $L$ with curvature $\iO{L,h_j}\geq\gamma$ in the sense of currents for some continuous real $(1,1)$-form $\gamma$ on $X$, which is not necessarily positive.
    \begin{itemize}
        \item [($a$)] We will write $h_1\preccurlyeq h_2$, and say that $h_1$ is \textit{less singular} than $h_2$, if there exists a constant $C > 0$ such that $h_1\leq Ch_2$.
        \item [($b$)] We will write $h_1 \sim h_2$, and say that $h_1$, $h_2$ are \textit{equivalent with respect to singularities}, if there exists a constant $C > 0$ such that $C^{-1}h_2\leq h_1\leq Ch_2$.
    \end{itemize}
\end{definition}

Finally, we introduce the strong openness property, which is also important for applications.

\begin{theorem}[\textnormal{Strong openness property, cf. \cite{GZ15}}]\label{Theorem: strong openness property}
    Let $\varphi$ be a negative plurisubharmonic function on $\Delta^n\subset\bb{C}^n$, and let $\psi\not\equiv-\infty$ be a negative plurisubharmonic function on $\Delta^n$. 
    Then we have $\scr{I}(\varphi)=\bigcup_{\varepsilon>0}\scr{I}(\varphi+\varepsilon\psi)$.
\end{theorem}

\subsection{Algebro-geometric positivity}

Let $X$ be a complex manifold which is not necessarily compact, $E\longrightarrow X$ be a holomorphic vector bundle and $V$ be a finite dimensional linear subspace of $H^0(X,E)$. 
We define the linear subsystem $|V|$ corresponding to $V$ by $|V| = \bb{P}(V)$. 
The base locus of the linear subsystem $|V|$ is given by 
\begin{align*}
    \rom{Bs}_{|V|}=\bigcap_{s\in V}s^{-1}(0)=\{x\in X\mid s(x)=0 \text{~ for ~ all ~} s\in V\setminus\{0\}\}.
\end{align*}

Let $U\subseteq X$ be an open subset of $X$.
The notions of \textit{generated by global sections} and \textit{semi}-\textit{ample} for a holomorphic vector bundle $E\longrightarrow X$ are known (see \cite[Chapter VII]{Dem-book}), and the corresponding notions \textit{on} $U$ are defined in the same way.
Here, if $E$ is a holomorphic line bundle, then being base-point free on $U$, being generated by global sections on $U$, and having $\rom{Bs}_{|E|}\cap U\ne\emptyset$ are equivalent.

Let $L\longrightarrow X$ be a holomorphic line bundle and $V$ be a finite dimensional linear subspace of $H^0(X,L)$.
Let $\sigma_0,\ldots,\sigma_N$ be a basis of the linear subsystem $V$, then we have the Kodaira map $\varPhi_V:X\dashrightarrow \bb{P}^N:x\longmapsto[\sigma_0(x):\cdots:\sigma_N(x)]$,
which is holomorphic on $X\setminus \rom{Bs}_{|V|}$, and that the isomorphism $L\cong\varPhi^*_V\cal{O}_{\bb{P}^N}(1)$ is obtained on $X\setminus \rom{Bs}_{|V|}$.

\begin{definition}[{cf. \cite{Tak98,Fuj75}}]
    Let $X$ be a complex manifold (not necessarily compact), $U\subseteq X$ be an open subset of $X$ and $L$ be a holomorphic line bundle on $X$. 
    \begin{itemize}
        \item $L$ is \textit{very ample on} $U$ if there exists a finite dimensional linear subspace $V$ of $H^0(U,L)$ such that the Kodaira map 
        $\varPhi_V:U\dashrightarrow \bb{P}^{\rom{dim}\,V-1}$ is holomorphic embedding on $U$, i.e., injective holomorphic immersion.
        \item $L$ is \textit{ample on} $U$ if there is an integer $m_U$ such that $L^{\otimes m_U}$ is very ample on $U$.
    \end{itemize}
\end{definition}

\begin{definition}[{\cite{Wat24}}]
    Let $X$ be a complex manifold (not necessarily compact) and $L\longrightarrow X$ be a holomorphic line bundle. Set 
    \begin{align*}
        \varrho_m:=\sup\{\rom{dim}\,\overline{\varPhi_V(X)}\mid V\subseteq H^0(X,L^{\otimes m}) \text{~with~} \rom{dim}\,V<+\infty\}
    \end{align*}
    if $H^0(X,L^{\otimes m})\ne\{0\}$, and $\varrho_m=-\infty$ otherwise. The \textit{Kodaira}-\textit{Iitaka} \textit{dimension} of $L$ is defined by $\kappa(L):=\max\{\varrho_m\mid m\in\bb{N}\}$. 
    We say that $L$ is \textit{big} if $\kappa(L)=\rom{dim}\,X$. 

    For an open subset $U\subseteq X$, the \textit{volume} of $L$ on $U$ is defined by 
    \begin{align*}
        \rom{Vol}_U(L):=n!\cdot\liminf_{k\to+\infty}\frac{\rom{dim}\,H^0(U,L^{\otimes k})}{k^n}.
    \end{align*}
\end{definition}

It is well known (see \cite{Dem90,Dem12}) that, if $X$ is compact, the following are equivalent for a line holomorphic bundle $L\longrightarrow X$; 
$L$ is big, $\mathrm{Vol}_X(L) > 0$, and $L$ admits a singular Hermitian metric which is singular positive on $X$. 

\section{exhaustion positivity}\label{Section: exhaustion positivity}

In this section, we provide the proofs of the following theorem and Theorem \ref{Theorem: exh sing positivity on tlX twisted by psef}.

\begin{theorem}\label{Theorem: exh sing positivity on tlX}
    Let $X$ be a weakly pseudoconvex complex space, $\pi:\tx\longrightarrow X$ be a canonical desingularization and $L\longrightarrow X$ be a holomorphic line bundle. 
    If $L$ is positive, then the pulled-back line bundle $\pi^*L\longrightarrow\tx$ is exhaustion singular-positive with $\cal{O}_{\tx}$. 
\end{theorem}

For this purpose, we first present a useful lemma.

\begin{lemma}\label{Lemma: key lemma}
    Let $X$ be a complex space, $\pi:\widetilde{X}\longrightarrow X$ be a canonical desingularization and $E:=\pi^{-1}(\xs)$ be the $\pi$-exceptional set which is simple normal crossing. 
    Let $V$ be a relatively compact open subsets of $X$, and set $\widetilde{V}:=\pi^{-1}(V)$.
    Then there exists a quasi-plurisubharmonic function $\psi:\widetilde{V}\longrightarrow[-\infty,+\infty)$ which is smooth on $\widetilde{V}\setminus E$, and whose decomposition $\psi=\psi_{\rom{sm}}+\psi_{\rom{psh}}$ in a neighborhood of $E$ satisfies that the smooth part $\psi_{\rom{sm}}$ has a Levi form $\idd\psi_{\rom{sm}}$ compensating for the loss of positivity of the pullback $\pi^*\omega$ on $\tl{V}$, where $\pi^*\omega$ degenerates along $E$ for any Hermitian metric $\omega$ on $X$.
    Hence, by the relative compactness of $V$, there exists $\varepsilon_V>0$ such that $\pi^*\omega+\varepsilon\idd\psi$ is strictly positive on $\tl{V}$ in the sense of currents for any $0<\varepsilon<\varepsilon_V$; in other words, there exists a Hermitian metric $\gamma_\varepsilon$ on $\tx$ such that 
    \begin{align*}
        \pi^*\omega+\varepsilon\idd\psi\geq\gamma_\varepsilon
    \end{align*}
    on $\tl{V}$ in the sense of currents.
    In particular, we have $\psi\in L^1_{loc}(\widetilde{V})$.
\end{lemma}

\begin{proof}
    By Negativity Lemma \ref{Lemma: Negativity lemma}, there exist an effective divisor $E_V=\sum_{j\in J_V}b_jE_j$ with $b_j\in\bb{N}$ and a smooth Hermitian metric $h^*_{E_V}$ on $\cal{O}_{\tx}(-E_V)|_{\tl{V}}$ such that the following holds: 
    for any Hermitian metric $\omega$ on $X$, there exists $\varepsilon_V>0$ such that 
    \begin{align*}
        \gamma_\varepsilon:=\pi^*\omega+\varepsilon\,\iO{\cal{O}_{\tx}(-E_V),h^*_{E_V}}>0
    \end{align*}
    on $\tl{V}$, for any $0<\varepsilon<\varepsilon_V$. 
    Let $\sigma_j$ be the defining section of $\cal{O}_{\tx}(-E_j)$ for each $j\in J_V$, and set $\sigma_V:=\bigotimes_{j\in J_V}\sigma_j^{b_j}$ as a section of $\cal{O}_{\tx}(E_V)$. 
    Then, we define the natural singular Hermitian metric $1/|\sigma_V|^2$ on $\cal{O}_{\tx}(E_V)$ by $1/|\sigma_V|^2:=h_{sm}/|\sigma_V|^2_{h_{sm}}$ for some smooth Hermitian metric $h_{sm}$ on $\cal{O}_{\tx}(E_V)$.
    It follows immediately that the metric $1/|\sigma_V|^2$ is singular semi-positive on $\tx$.

    Thus, the singular metric $h^*_{E_V}\otimes 1/|\sigma_V|^2$ on the trivial bundle $\cal{O}_{\tx}$ can be regarded as a function. 
    In particular, the weight function $\psi:=-\log (h^*_{E_V}\otimes 1/|\sigma_V|^2)$ is a globally defined locally integrable function on $\tl{V}$, i.e., $\psi\in L^1_{loc}(\tl{V})$.
    Locally, a computation yields 
    \begin{align*}
        \idd\psi=-\idd\log h^*_{E_V}+\idd\log|\sigma_V|^2\geq\iO{\cal{O}_{\tx}(-E_V),h^*_{E_V}}
    \end{align*}
    in the sense of currents. Since $\log |\sigma_V|$ is locally plurisubharmonic, the function $\psi$ is quasi-plurisubharmonic and smooth on $\tl{V}\setminus E$. 
    Hence, $\psi$ is the desired function and satisfies $\pi^*\omega+\varepsilon\idd\psi\geq\gamma_\varepsilon$ on $\tl{V}$ in the sense of currents.
\end{proof}

Theorem \ref{Theorem: exh sing positivity on tlX} follows immediately from the following proposition.

\begin{proposition}\label{Proposition: singular positivity on tlX_c by using SOP}
    Let $(X,\varPsi)$ be a weakly pseudoconvex complex space, 
    $\pi:\widetilde{X}\longrightarrow X$ be a canonical desingularization and $E:=\pi^{-1}(\xs)$ be a $\pi$-exceptional set which is simple normal crossing.
    Let $L\longrightarrow X$ be a holomorphic line bundle, 
    and take 
    $c>\inf_X\varPsi$. 
    
    For some $\tau>0$, if there exists a smooth Hermitian metric $h$ on $L|_{X_{c+\tau}}$ whose curvature is positive on $X_{c+\tau}$,
    then there exists a singular Hermitian metric $\tl{h}_c$ on $\pi^*L|_{\tl{X}_{c+\tau}}$ such that $\tl{h}_c$ is smooth on $\tl{X}_{c+\tau}\setminus E$, $\scr{I}(\tl{h}_c)=\cal{O}_{\tx}$ on $\tl{X}_c$ and $\tl{h}_c$ is singular positive on $\tx_c$, i.e., $\iO{\pi^*L,\tl{h}_c}\geq\gamma$ on $\tl{X}_c$ in the sense of currents for some Hermitian metric $\gamma>0$.
\end{proposition}

\begin{proof}
    For the relatively compact subset $\tl{X}_{c+2\tau}$, applying Lemma \ref{Lemma: key lemma}, there exists a quasi-plurisubharmonic function $\psi:\tl{X}_{c+2\tau}\longrightarrow[-\infty,+\infty)$ satisfying the conditions of Lemma \ref{Lemma: key lemma}. 
    Here, for a smooth Hermitian metric $\pi^*h$ on $\pi^*L$, the curvature $\iO{\pi^*L,\pi^*h}=\pi^*\iO{L,h}$ is positive on $\tx\setminus E$ and semi-positive on $\tx$. 
    From the relative compact-ness of $\tl{X}_{c+\tau}$, there exists an integer $m_c\in\bb{N}$ such that for any integer $m\geq m_c$, the singular Hermitian metric $\tl{h}_m:=\pi^*h^m\cdot e^{-\psi}$ on $\pi^*L^{\otimes m}|_{\tl{X}_{c+2\tau}}$ is singular positive on $\tl{X}_{c+\tau}$, i.e., 
    \begin{align*}
        \iO{\pi^*L^{\otimes m},\tl{h}_m}=m\cdot\pi^*\iO{L,h}+\idd\psi\geq m_c\cdot\pi^*\iO{L,h}+\idd\psi\geq\tl{\gamma},
    \end{align*}
    in the sense of currents on $\tl{X}_{c+\tau}$ for some Hermitian metric $\tl{\gamma}>0$.
    In other words, by Lemma \ref{Lemma: key lemma}, for any positive number $0<\varepsilon<1/m_c$, the singular Hermitian metric $\tl{h}_\varepsilon:=\pi^*h\cdot e^{-\varepsilon\psi}$ on $\pi^*L|_{\tl{X}_{c+2\tau}}$ is singular positive on $\tl{X}_{c+\tau}$.

    By Skoda's result \cite{Sko72}, if $\nu(-\log \widetilde{h}_\varepsilon,x)<2$ then $\scr{I}(\tl{h}_\varepsilon)=\cal{O}_{\tx,x}$ immediately. 
    From the equality $\nu(-\log\tl{h}_\varepsilon,x)=\nu(\varepsilon\psi,x)=\varepsilon\nu(\psi,x)$,
    let $M_{\psi,c}:=\max_{x\in\overline{\tl{X}}_{c+\tau}}\nu(\psi,x)$. 
    Thus, by choosing $\varepsilon_c>0$ such that $\varepsilon_c < \min\{1/m_c, 2/M_{\psi,c}\}$, the desired singular Hermitian metric $\tl{h}_c := \tl{h}_{\varepsilon_c}$ on $\pi^*L|_{\tl{X}_{c+\tau}}$ is obtained. 
\end{proof}

Furthermore, to prove the global embedding theorem, the following theorem is proved by using the strong openness property ($=$ Theorem \ref{Theorem: strong openness property}, see \cite{GZ15}).

\begin{theorem}\label{Theorem: exh sing positivity on tlX twisted by psef}
    Let $X$ be a weakly pseudoconvex complex space, $\pi:\tx\longrightarrow X$ be a canonical desingularization and $L\longrightarrow X$ be a holomorphic line bundle. 
    Let $P\longrightarrow \tx$ be a holomorphic line bundle equipped with a singular Hermitian metric $h_P$.
    If $L$ is positive and $h_P$ is singular semi-positive on $\tx$, then the holomorphic line bundle $P\otimes\pi^*L$ is exhaustion singular-positive with $\scr{I}(h_P)$. 

    In particular, we can construct a sequence of singular Hermitian metrics $\{\tl{h}_j\}_{j\in\mathbb{N}}$ defining $\pi^*L$ to be exhaustion singular-positive with $\cal{O}_{\tl{X}}$, 
    such that the sequence $\{\tl{h}_j \otimes h_P\}_{j\in\mathbb{N}}$ 
    defines $\pi^*L\otimes P$ to be exhaustion singular-positive with $\scr{I}(h_P)$.
\end{theorem}

\begin{proof}
    For any constant $c>\inf_X\varPsi$, it suffices to construct a singular Hermitian metric $\tl{h}_c$ of $\pi^*L$ on $\tl{X}_c$ which is singular positive on $\tl{X}_c$ and satisfies $\scr{I}(\tl{h}_c)=\cal{O}_{\tx}$ and $\scr{I}(\tl{h}_c\otimes h_P)=\scr{I}(h_P)$ on $\tl{X}_c$.

    For some $\tau > 0$, similarly to the proof of Proposition \ref{Proposition: singular positivity on tlX_c by using SOP}, there exist a quasi-plurisubharmonic function $\psi: \tl{X}_{c+2\tau}\longrightarrow [-\infty,+\infty)$ and an integer $m_c \in \mathbb{N}$ such that for any $0 < \varepsilon < 1/m_c$, 
    the singular Hermitian metric $\tl{h}_{\varepsilon}:=\pi^*h\cdot e^{-2\varepsilon\psi}$ on $\pi^*L|_{\tl{X}_{c+2\tau}}$ is singular positive on $\tl{X}_{c+\tau}$. In particular, $\tl{h}_{\varepsilon}\otimes h_P$ is also singular positive on $\tl{X}_{c+\tau}$. 

    Let $H_{sm}$ be a smooth Hermitian metric on $P$, then $\tl{\varphi}:=-\frac{1}{2}\log(h_P/H_{sm})$ is a locally integrable function on $\tl{X}$, i.e., $\tl{\varphi}\in L^1_{loc}(\tx)$. 
    By the singular semi-positivity of $h_P$, there exists a quasi-plurisubharmonic function $\varphi$ on $\tx$ such that $\tl{\varphi}=\varphi$ almost everywhere. 
    Here, we can replace $h_P$ by $h_P:=H_{sm}\cdot e^{-2\varphi}$ without loss of generality, and we have $\scr{I}(\tl{h}_\varepsilon)=\scr{I}(\varepsilon\psi)$ and $\scr{I}(\tl{h}_\varepsilon\otimes h_P)=\scr{I}(\varphi+\varepsilon\psi)$ on $\tl{X}_{c+2\tau}$, and $\scr{I}(h_P)=\scr{I}(\varphi)$ on $\tx$.

    By the strong openness property (=\,Theorem \ref{Theorem: strong openness property}), we have $\scr{I}(\varphi)=\bigcup_{0<\delta<1/m_c}\scr{I}(\varphi+\delta\psi)$ on $\tl{X}_{c+\tau}$. 
    Therefore, by the relative compact-ness of $\tl{X}_c$ and the strong Noetherian property of coherent sheaves (see \cite[Chapter\,II, (3.22)]{Dem-book}), there exists $0<\delta_c<1/m_c$ such that $\bigcup_{0<\delta<1/m_c}\scr{I}(\varphi+\delta\psi)=\scr{I}(\varphi+\delta_c\psi)$ on $\tl{X}_c$.  
    Hence, for any $0<\varepsilon\leq\delta_c$, we have $\scr{I}(\varphi)=\scr{I}(\varphi+\varepsilon\psi)$, i.e., $\scr{I}(\tl{h}_\varepsilon\otimes h_P)=\scr{I}(h_P)$, on $\tl{X}_c$. 
    Let $M_{\psi,c}:=\max_{x\in\overline{\tl{X}}_{c+\tau}}\nu(\psi,x)$. 
    Similarly to the proof of Proposition \ref{Proposition: singular positivity on tlX_c by using SOP}, by choosing $\varepsilon_c>0$ such that $\varepsilon_c<\min\{\delta_c,1/M_{\psi,c}\}$, 
    the desired singular Hermitian metric $\tl{h}_c:=\tl{h}_{\varepsilon_c}=\pi^*h\cdot e^{-2\varepsilon_c\psi}$ on $\pi^*L|_{\tl{X}_{c+\tau}}$ is obtained.
\end{proof}

\begin{remark}
    The sequence of singular metrics $\{\tl{h}_j\}_{j \in \mathbb{N}}$ that makes the line bundle $\pi^*L$ exhaustion positive as obtained in Theorem \ref{Theorem: exh sing positivity on tlX} can, in general, be constructed such that $\tl{h}_{j+1}|_{\tl{X}_j}\preccurlyeq\tl{h}_j$, i.e., there exists a constant $C_{j+1}>0$ such that $\tl{h}_{j+1} \leq C_{j+1} \tl{h}_j$ on ${\tl{X}_j}$.
\end{remark}

\begin{proof}
    We construct the sequence inductively. Fix $\gamma\in\mathbb{N}$. Assume $\tl{h}_\gamma$ has been constructed as in Proposition \ref{Proposition: singular positivity on tlX_c by using SOP}. 
    Then we construct $\tl{h}_{\gamma+1}$ such that $\tl{h}_{\gamma+1}|_{\tl{X}_\gamma}\preccurlyeq\tl{h}_\gamma$ holds.

    The singular Hermitian metric $\tl{h}_\gamma$ was constructed as follows:
    By Negativity Lemma \ref{Lemma: Negativity lemma}, there exist an effective divisor $E_{X_\gamma}=\sum_{j\in J_\gamma}b_jE_j$ with $b_j\in\bb{N}$ and a smooth Hermitian metric $h^*_{E_{X_\gamma}}$ on $\cal{O}_{\tx}(-E_{X_\gamma})|_{X_{\gamma+1/2}}$ whose curvature 
    compensates for the loss of positivity of the pull-back Hermitian metric along $E$.
    Here $h^*_{E_{X_\gamma}}=\bigotimes_{j\in J_\gamma}{h^*_{E_j}}^{\!\!\otimes b_j}$, and the line bundle $\cal{O}_{\tx}(-E_j)|_{E_j\cap \tl{X}_{\gamma+1/2}}$ coincides with the tautological line bundle of the normal bundle, and each $h^*_{E_j}$ is the smooth metric defined on $\tl{X}_{\gamma+1/2}$ and induced by the Fubini-Study metric on that bundle. 
    Let $\sigma_j$ be the defining section of $\cal{O}_{\tx}(-E_j)$ for each $j\in J_\gamma$, and set $\sigma_{E_\gamma}:=\bigotimes_{j\in J_\gamma}\sigma_j^{b_j}$ as a section of $\cal{O}_{\tx}(E_{X_\gamma})$. 
    Let $\psi_\gamma:=-\log(h^*_{E_{X_\gamma}}\otimes 1/|\sigma_{E_\gamma}|^2)$ and $M_{\psi_\gamma}:=\max_{x\in\overline{\tl{X}}_{\gamma+1/2}}\nu(\psi_\gamma,x)$. Here, there exists $m_\gamma\in\bb{N}$ such that $\tl{h}_{\gamma,\varepsilon}:=\pi^*h\cdot e^{-\varepsilon\psi_\gamma}$ is singular positive on $\tl{X}_{\gamma+1/2}$ for any $0<\varepsilon<1/m_\gamma$. 
    Thus, we choose $\varepsilon_\gamma>0$ such that $\varepsilon_\gamma < \min\{1/m_\gamma,2/M_{\psi_\gamma}\}$, and $\tl{h}_\gamma$ is then constructed by $h_\gamma := \pi^*h\cdot e^{-\varepsilon_\gamma\psi_\gamma}$.

    We then construct the singular Hermitian metric $\tl{h}_{\gamma+1}$.
    By Negativity Lemma \ref{Lemma: Negativity lemma}, as above, there exist an effective divisor $E_{X_{\gamma+1}}=\sum_{j\in J_{\gamma+1}}\beta_jE_j$ with $\beta_j\in\bb{N}$ and a smooth Hermitian metric $\hbar^*_{E_{X_{\gamma+1}}}$ on $\cal{O}_{\tx}(-E_{X_{\gamma+1}})|_{X_{\gamma+3/2}}$ whose curvature compensates for the loss of positivity of the pull-back Hermitian metric along $E$.
    Here $\hbar^*_{E_{X_\gamma}}=\bigotimes_{j\in J_{\gamma+1}}{\hbar^*_{E_j}}^{\!\!\otimes \beta_j}$, and since $\pi$ is canonical (see Theorem \ref{Theorem: canonical desingularization}), for each $j \in \{\ell\in J_\gamma\mid E_\ell\cap(\tl{X}_{\gamma+3/2}-\overline{\tl{X}}_{\gamma+1/2})\ne\emptyset\}$, the restriction of $\mathcal{O}(-E_j)|_{\tl{X}_{\gamma+3/2}}$ to $\tl{X}_{\gamma+1/2}$ coincides with $\mathcal{O}(-E_j)|_{\tl{X}_{\gamma+1/2}}$, which provides $\hbar_{E_j}$ as a natural extension of $h_{E_j}$, i.e., $\hbar_{E_j}|_{\tl{X}_{\gamma+1/2}}=h_{E_j}$.
    From the proof of the Negative Lemma \ref{Lemma: Negativity lemma} (see \cite[Lemma 3.6]{Wat24}), for any $\beta_j$ with $j \in J_\gamma$, there exists an integer $\kappa_\gamma$ such that $\beta_j = \kappa_\gamma b_j$ for all $j \in J_\gamma$. Hence, we have that $E_{X_{\gamma+1}}|_{\tl{X}_{\gamma+1/2}}=\kappa_\gamma E_{X_\gamma}$ and $\hbar^*_{E_{X_{\gamma+1}}}|_{\tl{X}_{\gamma+1/2}}={h^*_{E_{X_\gamma}}}^{\!\!\!\!\!\!\otimes\kappa_\gamma}$.

    Let $\sigma_j$ be the defining section of $\cal{O}_{\tx}(-E_j)$ for each $j\in J_{\gamma+1}$, and set $\sigma_{E_{\gamma+1}}:=\bigotimes_{j\in J_{\gamma+1}}\sigma_j^{\beta_j}$ as a section of $\cal{O}_{\tx}(E_{X_{\gamma+1}})$. 
    Let $\psi_{\gamma+1}:=-\log(\hbar^*_{E_{X_{\gamma+1}}}\otimes 1/|\sigma_{E_{\gamma+1}}|^2)$ and $M_{\psi_{\gamma+1}}:=\max_{x\in\overline{\tl{X}}_{\gamma+3/2}}\nu(\psi_{\gamma+1},x)$, then we obtain $\psi_{\gamma+1}|_{\tl{X}_{\gamma+1/2}}=\kappa_\gamma\psi_\gamma$. 
    Here, there exists $m_{\gamma+1}\in\bb{N}$ such that $\kappa_\gamma m_\gamma\leq m_{\gamma+1}$ and $\tl{h}_{\gamma+1,\varepsilon}:=\pi^*h\cdot e^{-\varepsilon\psi_{\gamma+1}}$ is singular positive on $\tl{X}_{\gamma+3/2}$ for any $0<\varepsilon<1/m_{\gamma+1}$. 
    Thus, we choose $\varepsilon_{\gamma+1}>0$ such that $\varepsilon_{\gamma+1} < \min\{1/m_{\gamma+1},2/M_{\psi_{\gamma+1}},\varepsilon_\gamma/\kappa_\gamma\}$, and the desired singular Hermitian metric $\tl{h}_{\gamma+1}$ is then constructed by $h_{\gamma+1} := \pi^*h\cdot e^{-\varepsilon_{\gamma+1}\psi_{\gamma+1}}$. 

    In fact, the ordering of singular Hermitian metrics is determined by the rate at which their weight functions tend to $-\infty$. 
    Hence, we may assume $\psi_\tau < 0$ with $\tau=\gamma,\gamma+1$, and the construction yields the inequality 
    \begin{align*}
        0>\varepsilon_{\gamma+1}\psi_{\gamma+1}|_{\tl{X}_{\gamma+1/2}}=\varepsilon_{\gamma+1}\kappa_p\psi_\gamma>\varepsilon_\gamma\psi_\gamma,
    \end{align*}
    which expresses $\tl{h}_{j+1}|_{\tl{X}_j}\preccurlyeq\tl{h}_j$.
\end{proof}

\begin{example}\label{Example: more singular example}
    Let $(X,\varPsi)$ be a weakly pseudoconvex complex space and $L\longrightarrow X$ be a positive line bundle.
    If there exists an increasing sequence of real numbers $\{c_j\}_{j\in\mathbb{N}}$ diverging to $+\infty$ such that $\partial X_{c_j}\cap \xs=\emptyset$ for any $j\in\mathbb{N}$, 
    then there exists a sequence of singular Hermitian metrics $\{\tl{h}_j\}_{j\in\bb{N}}$ defining the exhaustion singular-positive with $\cal{O}_{\tl{X}}$ of $\pi^*L$, 
    which becomes progressively more singular, i.e., $\tl{h}_j\preccurlyeq \tl{h}_{j+1}|_{\tl{X}_{c_j}}$ for any $j\in \bb{N}$. 
    
    In particular, the sequence can also be arranged to be equivalent with respect to singularities, that is, $\tl{h}_j\sim \tl{h}_{j+1}|_{\tl{X}_{c_j}}$ for any $j\in \bb{N}$. 
    In this case, we can construct a globally defined singular Hermitian metric $\tl{h}$ on $\pi^*L$ which is singular positive on $\tx$ and satisfies $\scr{I}(\tl{h})=\cal{O}_{\tx}$ on $\tl{X}$.
\end{example}

\begin{proof}
    By assumption, $\xs$ is separated by each boundary $\partial X_{c_j}$.
    Thus, setting $\xs^j := (X_{c_j} \cap \xs)\setminus X_{c_{j-1}}$ for any $j\in\bb{N}$, here $X_{c_0}:=\emptyset$, it suffices to treat $\xs^j$ and $\xs$ separately. 
    Setting $E_j:=\pi^{-1}(\xs^j)$, the assumption implies that $E_j\Subset \tx_{c_j}\setminus\tx_{c_{j-1}}$.
    For each $j\in\bb{N}$, by Lemma \ref{Lemma: key lemma}, there exists a quasi-plurisubharmonic function $\psi_j:\tl{X}_{c_j}\longrightarrow[-\infty,+\infty)$ which is smooth on $\tl{X}_{c_j}\setminus E_j$.
    We can choose an open neighborhood $\tl{U}_j$ of $E_j$ such that $E_j\Subset\tl{U}_j\Subset \tx_{c_j}\setminus\tx_{c_{j-1}}$.
    Let $\chi_j:\tl{X}\longrightarrow\bb{R}_{\geq0}$ be a cut-off function for $\tl{U}_j$; that is, $\chi_j$ is smooth on $\tl{X}$ satisfying $\chi_j \ge 0$, $\chi = 1$ on $\tl{U}_j$, and $\chi_j = 0$ on $\tl{X} \setminus (\overline{\tl{X}_{c_j}\setminus\tl{X}_{c_{j-1}}})$.
    Let $\tl{\psi}_j := \chi_j \psi_j$, then the quasi-plurisubharmonic function $\tl{\psi}_j$ is defined on $\tx$, is smooth on $\tl{X} \setminus E_j$, and vanishes on $\tl{X}\setminus(\overline{\tl{X}_{c_j}\setminus\tl{X}_{c_{j-1}}})$. 

    By the relative compactness of $\tx_{c_j}\setminus\tl{X}_{c_{j-1}}$, similarly to the proof of Proposition \ref{Proposition: singular positivity on tlX_c by using SOP}, there exists $\varepsilon_j > 0$ such that for any $0 < \varepsilon \leq \varepsilon_j$, the singular Hermitian metric $\pi^*h\cdot e^{-\varepsilon\tl{\psi}_j}$ is singular positive on $\tx_{c_j}\setminus\tl{X}_{c_{j-1}}$ and satisfies $\scr{I}(\pi^*h\cdot e^{-\varepsilon\tl{\psi}_j})=\cal{O}_{\tx}$.
    Take a decreasing sequence of real numbers $\{\varepsilon_j(k)\}_{k \in \bb{N}}$ converging to $0$, with $\varepsilon_j(1) = \varepsilon_j$. 
    From the above, by setting the singular Hermitian metric of $\pi^*L$ on $\tl{X}_{c_j}$ as 
    \begin{align*}
        \displaystyle \tl{h}_j:=\pi^*h\cdot e^{-\sum^j_{k=1}\varepsilon_k(j+1-k)\tl{\psi}_k}
    \end{align*}
    we obtains the desired sequence of singular Hermitian metrics $\{\tl{h}_j\}_{j\in\bb{N}}$, which satisfy $\tl{h}_j\preccurlyeq \tl{h}_{j+1}|_{\tl{X}_{c_j}}$ and $\scr{I}(\tl{h}_j)=\cal{O}_{\tx}$ on $\tl{X}_{c_j}$ for any $j\in\bb{N}$.

    In particular, by setting the singular Hermitian metric of $\pi^*L$ on $\tl{X}_{c_j}$ as $\tl{h}_j:=\pi^*h\cdot \exp(-\sum^j_{k=1}\varepsilon_k\tl{\psi}_k)$, we obtains a sequence $\{\tl{h}_j\}_{j\in\bb{N}}$ that is equivalent with respect to singularities.
    Furthermore, by defining a global singular Hermitian metric on $\tl{X}$ as $\tl{h} = \pi^*h\cdot \exp(-\sum^{+\infty}_{k=1}\varepsilon_k\tl{\psi}_k)$, this metric is singular positive on $\tl{X}$ and satisfies $\scr{I}(\tl{h}) = \cal{O}_{\tx}$ on $\tl{X}$.
\end{proof}

\section{Refined Runge-type approximation theorems}\label{Section: Approximation theorems}

In this section, the proof of Theorem \ref{Theorem: ext approximation theorem} is given. 
We first introduce the semi-norm  $|\bullet|_K$. 
Let $X$ be a weakly pseudoconvex manifold, $\gamma$ be a Hermitian metric on $X$ and $h_0$ be any smooth Hermitian metric of $L$ on $X$. 
For a compact subset $K$ in $X_c$, we put 
\begin{align*}
    |\phi|_K:=\sup_{x\in K}|\phi|_{h_0,\gamma}(x)
\end{align*}
for $\phi\in H^0(X_c,K_X\otimes L)$. Then $\{|\bullet|_K\}_{K}$ gives a system of semi-norms in $H^0(X_c,K_X\otimes L)$.

\begin{proposition}[\textnormal{Runge-type approximation property for $L^2$-norms and semi-norms, see \cite[Theorem\,1.1,\,Proposition\,4.8 and 4.10]{Wat24}}]\label{Proposition: approximation thm with L2-norms and semi-norms}
    Let $(X,\varPsi)$ be a weakly pseudoconvex manifold and $L\longrightarrow X$ be a holomorphic line bundle, and take arbitrary numbers $c>b>\inf_X\varPsi$.
    For some $\tau>0$, if there exists a singular Hermitian metric $h$ on $L|_{X_{c+\tau}}$ which is singular positive on $X_{c+\tau}$, then the restriction map 
    \begin{align*}
        \rho_{c,b}:H^0(X_c,K_X\otimes L\otimes\scr{I}(h))\longrightarrow H^0(X_b,K_X\otimes L\otimes\scr{I}(h))
    \end{align*}
    has a dense images with respect to the topology induced by the $L^2$-norm $||\bullet||_h$ and the topology of uniform convergence on all compact subsets in $X_b$.
    
    In other words, for any compact subset $K$ in $X_b$, any positive number $\varepsilon>0$ and any holomorphic section $\phi\in H^0(X_b,K_X\otimes L\otimes\scr{I}(h))$, 
    there exists $\widetilde{\phi}\in H^0(X_c,K_X\otimes L\otimes\scr{I}(h))$ such that $||\tl{\phi}-\phi||_{h,X_b}<\varepsilon$ and $|\widetilde{\phi}-\phi|_K<\varepsilon$.
\end{proposition}

\begin{proof}[Proof of Theorem \ref{Theorem: ext approximation theorem}]
    Let $\gamma:=\lfloor c\rfloor+2$, where $\lfloor\bullet\rfloor$ denotes the integer part. 
    Here, there exists a sequence of singular Hermitian metrics $\{h_j\}_{j\in\bb{N}}$ satisfying the appropriate conditions. 
    For simplicity, we may assume that each $h_j$ is defined on $L|_{X_j}$ without loss of generality; that is, we may take the sequence $\{c_j\}_{j\in\bb{N}}$ as $c_j = j$.
    By the Runge-type approximation Proposition \ref{Proposition: approximation thm with L2-norms and semi-norms} and the exhaustion positivity of $h$, the restriction map 
    \begin{align*}
        \rho_k:H^0(X_{c+k+1},K_X\otimes L\otimes\scr{I}(h_{\gamma+k}))\longrightarrow H^0(X_{c+k},K_X\otimes L\otimes\scr{I}(h_{\gamma+k}))
    \end{align*}
    has dence image with the topology induced by the $L^2$-norm $||\bullet||_{h_{\gamma+k}}$ and the topology of uniform convergence on all compact subsets in $X_{c+k}$ for every $k=0,1,2,\ldots$. 
    Hence, for any compact subset $K$ in $X_c$, any $\varepsilon>0$ and any $\phi\in H^0(X_c,K_X\otimes L\otimes\scr{I}(h_\gamma))$, we can find a sequence $\phi_k\in H^0(X_{c+k},K_X\otimes L\otimes\scr{I}(h_{\gamma+k}))$ such that 
    \begin{align*}
        |\phi_1-\phi|_K<\frac{\varepsilon}{2} \quad \rom{and} \quad |\phi_{k+1}-\phi_k|_{\overline{X}_{c+k-1}}<\frac{\varepsilon}{2^k}.
    \end{align*}

    We define the section $\widetilde{\phi}$ of $K_X\otimes L$ on $X$ by
    \begin{align*}
        \widetilde{\phi}:=\phi_1+\sum^\infty_{k=1}(\phi_{k+1}-\phi_k)=\phi_\mu+\sum^\infty_{k=\mu}(\phi_{k+1}-\phi_k).
    \end{align*}
    For each $\mu\in\bb{N}\cup\{0\}$, by the following inequality
    \begin{align*}
        |\widetilde{\phi}-\phi_{\mu+1}|_{\overline{X}_{c+\mu}}\leq\sum^\infty_{k=\mu+1}|\phi_{k+1}-\phi_k|_{\overline{X}_{c+\mu}}<\sum^\infty_{k=\mu+1}\frac{\varepsilon}{2^k}=\frac{\varepsilon}{2^\mu}
    \end{align*}
    on $\overline{X}_{c+\mu}$, the holomorphic sequence $\{\phi_{\mu+\ell}\}_{\ell\in\bb{N}\cup\{0\}}$ converges to $\widetilde{\phi}$ uniformly on all compact subset of $X_{c+\mu}$.
    Hence, the section $\widetilde{\phi}$ is holomorphic on each $X_{c+\mu}$, and we obtain $\widetilde{\phi}\in H^0(X,K_X\otimes L)$. 
    Furthermore, we derive the density $|\widetilde{\phi}-\phi|_K<\varepsilon$ by setting $\mu=1$.
\end{proof}

The following corollary immediately follow from Theorems \ref{Theorem: exh sing positivity on tlX} and \ref{Theorem: ext approximation theorem}.

\begin{corollary}\label{Corollary: approximation theorem on tlX}
    Let $(X,\varPsi)$ be a weakly pseudoconvex complex space of pure dimension $n$ and $\pi:\tx\longrightarrow X$ be a canonical desingularization as in Theorem \ref{Theorem: canonical desingularization}. 
    If a holomorphic line bundle $L\longrightarrow X$ is positive, then for any sublevel set $\tx_c:=\pi^{-1}(X_c)=\{x\in\tx\mid\pi^*\varPsi(x)<c\}$, the natural restriction map
    \begin{align*}
        \rho_{\tx,c}:H^0(\tx,K_{\tx}\otimes\pi^*L)\longrightarrow H^0(\tx_c,K_{\tx}\otimes\pi^*L)
    \end{align*}
    has dense images with respect to 
    the topology of uniform convergence on all compact subsets in $\tx_c$.
\end{corollary}

\begin{proof}[Proof of Theorem \ref{Theorem: ext approximation theorem to complex spaces}]
As is well known, the Grauert-Riemenschneider canonical sheaf $\ogr$ is defined by $\Gamma(U,\ogr):=\{s\in\Gamma(U\cap\reg,\omega_{\reg})\mid i^{n^2}s\wedge\overline{s}\in L^1_{loc}(U)\}$ for any open subset $U\subset X$, and hance satisfies $\ogr=\pi_*\omega_{\tx}$. 
Therefore, the canonical desingularization $\pi:\tx\longrightarrow X$ induces an isomorphism 
\begin{align*}
        H^0(U,\ogr\otimes\cal{O}_X(L))\cong H^0(\pi^{-1}(U),K_{\tx}\otimes\pi^*L)
\end{align*}
via the pull-back $\pi^*$ and the push-forward $(\pi|_{\tx\setminus E})_*$, for any open subset $U\subseteq X$. 
Here, $E$ denotes the $\pi$-exceptional divisor; in fact, by the definition of $\ogr$ and the $\dbar$-extension lemma (see \cite[Chapter VIII, Lemma 7.3]{Dem-book}), the morphisms $\pi^*$ and $(\pi|_{\tx\setminus E})_*$ yield the isomorphism.

By combining this isomorphism with the natural restriction morphism $\rho_{\tx,c}$ in Corollary \ref{Corollary: approximation theorem on tlX}, we define the natural restriction map $\rho^\pi_c:H^0(X,\ogr\otimes\cal{O}_X(L))\longrightarrow H^0(X_c,\ogr$ $\otimes\cal{O}_X(L))$ define by $\rho^\pi_c:=(\pi|_{\tx\setminus E})_*\circ\rho_{\tx,c}\circ\pi^*$, i.e., by the following commutative diagram.
\[
\begin{CD}
H^0(\tx,K_{\tx}\otimes\pi^*L)
    @>{\rho_{\tx,c}}>>
H^0(\tx_c,K_{\tx}\otimes\pi^*L) \\
@V{\cong}VV @VV{\cong}V \\
H^0(X,\ogr\otimes\cal{O}_X(L))
    @>{\rho^\pi_c}>>
H^0(X_c,\ogr\otimes\cal{O}_X(L))
\end{CD}
\]

    We define appropriate semi-norms on the analytic space $X$.
    Let $\gamma_{\tx}$ be a Hermitian metric on $\tx$ and $h_0$ be any smooth Hermitian metric of $\pi^*L$ on $\tx$. 
    For a compact subset $K$ in $X_c$, we put 
    \begin{align*}
        |\phi|_{\pi,K}:=\sup_{x\in \pi^{-1}(K)}|\pi^*\phi|_{h_0,\omega}(x)=|\pi^*\phi|_{\pi^{-1}(K)}
    \end{align*}
    for any section $\phi\in H^0(X_c,\ogr\otimes\cal{O}_X(L))$. 
    Here, by the $\dbar$-extension lemma (see \cite[Chapter VIII, Lemma 7.3]{Dem-book}), the section $\pi^*\phi$ defined on $\tx_c \setminus E$ extends uniquely to a section on $\tx_c$. 
    For simplicity, the extended section is denoted by the same symbol $\pi^*\phi$.
    Then $\{|\bullet|_{\pi,K}\}_K$ gives a system of semi-norms in $H^0(X_c,\ogr\otimes\cal{O}_X(L))$.  

    By Corollary \ref{Corollary: approximation theorem on tlX}, for any positive number $\varepsilon>0$ and any holomorphic section $\phi\in H^0(X_c,\ogr\otimes\cal{O}_X(L))$, the pull-back $\pi^*\phi$ belongs to $H^0(\tx_c,K_{\tx}\otimes\pi^*L)$, 
    and there exists a holomorphic section $\tl{\psi}\in H^0(\tx,K_{\tx}\otimes\pi^*L)$ such that $|\rho_{\tx,c}(\tl{\psi})-\pi^*\phi|_{\pi^{-1}(K)}=|\tl{\psi}-\pi^*\phi|_{\pi^{-1}(K)}<\varepsilon$.
    Let $\tl{\phi}$ denote the image of $\tl{\psi}$ under the isomorphism induced by $(\pi|_{\tx\setminus E})_*$. Then $\tl{\phi}:=(\pi|_{\tx\setminus E})_*\tl{\psi} \in H^0(X,\ogr\otimes\cal{O}_X(L))$ is the desired holomorphic section satisfying 
    \begin{align*}
        |\rho^\pi_c(\tl{\phi})-\phi|_{\pi,K}=|(\pi|_{\tx\setminus E})_*\circ\rho_{\tx,c}(\tl{\psi})-\phi|_{\pi,K}=|\tl{\psi}-\pi^*\phi|_{\pi^{-1}(K)}<\varepsilon,
    \end{align*}
    which yields density.
\end{proof}

\begin{remark-theorem}\label{Remark-Theorem:}
    Concerning Runge-type approximation Theorem \ref{Theorem: ext approximation theorem}, if the sequence of singular Hermitian metrics on $L$ obtained from its exhaustion singular-positivity becomes progressively more singular, 
    then the space $H^0(X,K_X\otimes L\otimes\scr{J})$, containing the ideal sheaf $\scr{J}$, has a dense image under the natural restriction map $\rho_c$.

\textit{
    In other words, let $X$ be a weakly pseudoconvex manifold and $L$ be an exhaustion singular-positive with $\scr{J}$, 
    i.e., there exist an increasing sequence of real numbers $\{c_j\}_{j\in\mathbb{N}}$ diverging to $+\infty$ and a sequence of singular Hermitian metrics $\{h_j\}_{j\in\bb{N}}$ such that each $h_j$ is defined on $L|_{X_{c_j}}$, is singular positive on $X_{c_j}$, and satisfies $\scr{I}(h_j)=\scr{J}$ on $X_{c_j}$. 
    If $h_j\preccurlyeq h_{j+1}|_{X_{c_j}}$ for any $j\in\bb{N}$, then for any sublevel set $X_c$, the natural restriction map
    \begin{align*}
        \rho_c:H^0(X,K_X\otimes L\otimes\scr{J})\longrightarrow H^0(X_c,K_X\otimes L\otimes\scr{J})
    \end{align*}
    has dense images with respect to 
    the topology of uniform convergence on all compact subsets in $X_c$.
}

\textit{
    In particular, it is not necessary that $\scr{I}(h_j)=\scr{J}$ on $X_{c_j}$ for any $j\in\bb{N}$; 
    it suffices that there exists $j_0$ with $c_{j_0} > c$ such that $\scr{I}(h_{j_0})=\scr{J}$ on $X_{c_{j_0}}$, and $\scr{I}(h_j) \subseteq \scr{J}$ on $X_{c_j}$ for all $j > j_0$.
}
\end{remark-theorem}

\begin{proof}
    Let $\gamma:=\lfloor c\rfloor+2$, where $\lfloor\bullet\rfloor$ denotes the integer part. 
    For simplicity, we may assume that each $h_j$ is defined on $L|_{X_j}$ without loss of generality; that is, we may take the sequence $\{c_j\}_{j\in\bb{N}}$ as $c_j = j$.
    By the assumption, for any $j\in\bb{N}$ there exists a constant $C_{j+1}>0$ such that $h_j\leq C_{j+1}h_{j+1}|_{X_j}$.
    By the Runge-type approximation Proposition \ref{Proposition: approximation thm with L2-norms and semi-norms} and the exhaustion positivity of $h$, the restriction map 
    \begin{align*}
        \rho_k:H^0(X_{c+k+1},K_X\otimes L\otimes\scr{I}(h_{\gamma+k}))\longrightarrow H^0(X_{c+k},K_X\otimes L\otimes\scr{I}(h_{\gamma+k}))
    \end{align*}
    has dence image with the topology induced by the $L^2$-norm $||\bullet||_{h_{\gamma+k}}$ and the topology of uniform convergence on all compact subsets in $X_{c+k}$ for every $k=0,1,2,\ldots$. 
    Hence, for any compact subset $K$ in $X_c$, any $\varepsilon>0$ and any $\phi\in H^0(X_c,K_X\otimes L\otimes\scr{I}(h_\gamma))$, we can find a sequence $\phi_k\in H^0(X_{c+k},K_X\otimes L\otimes\scr{I}(h_{\gamma+k}))$ such that 
    \begin{align*}
        |\phi_1-\phi|_K<\frac{\varepsilon}{2}&, \quad\quad |\phi_{k+1}-\phi_k|_{\overline{X}_{c+k-1}}<\frac{\varepsilon}{2^k},\\
        ||\phi_1-\phi||_{h_\gamma,X_c}<\frac{\varepsilon}{2}& \quad \rom{and} \quad ||\phi_{k+1}-\phi_k||_{h_{\gamma+k},X_{c+k}}<\Big(\prod_{q=1}^{k}C_{\gamma+q}\Big)^{-1}\frac{\varepsilon}{2^k}.
    \end{align*}

    We define the section $\widetilde{\phi}$ of $K_X\otimes L$ on $X$ by
    \begin{align*}
        \widetilde{\phi}:=\phi_1+\sum^\infty_{k=1}(\phi_{k+1}-\phi_k)=\phi_\mu+\sum^\infty_{k=\mu}(\phi_{k+1}-\phi_k).
    \end{align*}
    For each $\mu\in\bb{N}\cup\{0\}$, by the following inequality
    \begin{align*}
        |\widetilde{\phi}-\phi_{\mu+1}|_{\overline{X}_{c+\mu}}\leq\sum^\infty_{k=\mu+1}|\phi_{k+1}-\phi_k|_{\overline{X}_{c+\mu}}<\sum^\infty_{k=\mu+1}\frac{\varepsilon}{2^k}=\frac{\varepsilon}{2^\mu}
    \end{align*}
    on $\overline{X}_{c+\mu}$, the holomorphic sequence $\{\phi_{\mu+\ell}\}_{\ell\in\bb{N}\cup\{0\}}$ converges to $\widetilde{\phi}$ uniformly on all compact subset of $X_{c+\mu}$.
    Hence, the section $\widetilde{\phi}$ is holomorphic on each $X_{c+\mu}$, and we obtain $\widetilde{\phi}\in H^0(X,K_X\otimes L)$. 
    Furthermore, we derive the density $|\widetilde{\phi}-\phi|_K<\varepsilon$ by setting $\mu=1$.

    For each $\mu\in\bb{N}$, the conditions obtained yield the inequality
    \begin{align*}
        ||\widetilde{\phi}-\phi_\mu||_{h_{\gamma+\mu},X_{c+\mu}}&\leq\sum_{k=\mu}^{\infty}||\phi_{k+1}-\phi_k||_{h_{\gamma+\mu},X_{c+\mu}}=\sum_{\ell=0}^{\infty}||\phi_{\mu+\ell+1}-\phi_{\mu+\ell}||_{h_{\gamma+\mu},X_{c+\mu}}\\
        &\leq\sum_{\ell=0}^{\infty}\Big(\prod_{q=1}^{\ell}C_{\gamma+\mu+q}\Big)||\phi_{\mu+\ell+1}-\phi_{\mu+\ell}||_{h_{\gamma+\mu+\ell},X_{c+\mu}}\\
        &\leq\sum_{\ell=0}^{\infty}\Big(\prod_{q=1}^{\ell}C_{\gamma+\mu+q}\Big)\Big(\prod_{q=1}^{\mu+\ell}C_{\gamma+q}\Big)^{-1}\frac{\varepsilon}{2^{\mu+\ell}}\\
        &=\Big(\prod_{q=1}^{\mu}C_{\gamma+q}\Big)^{-1}\sum_{\ell=0}^{\infty}\frac{\varepsilon}{2^{\mu+\ell}}
        =\Big(\prod_{q=1}^{\mu}C_{\gamma+q}\Big)^{-1}\frac{\varepsilon}{2^{\mu-1}}.
    \end{align*}
    This yields the $L^2$-inequality
    \begin{align*}
        \biggl(\int_{X_{c+\mu}}|\widetilde{\phi}|^2_{h_{\gamma+\mu},\omega_X}dV_{\omega_X}\biggr)^{\frac{1}{2}}=||\widetilde{\phi}||_{h_{\gamma+\mu},X_{c+\mu}}<||\phi_\mu||_{h_{\gamma+\mu},X_{c+\mu}}+\Big(\prod_{q=1}^{\mu}C_{\gamma+q}\Big)^{-1}\frac{\varepsilon}{2^{\mu-1}}<+\infty,
    \end{align*}
    for some Hermitian metric $\omega_X$ on $X$. 
    Therefore, we have $\widetilde{\phi}\in K_{X,x}\otimes L_x\otimes\scr{J}_x$ for any point $x$ in each sublevel set $X_{c+\mu}$, and thus $\widetilde{\phi}\in H^0(X,K_X\otimes L\otimes\scr{J})$ is obtained.
\end{proof}

From this Remark-Theorem and Example \ref{Example: more singular example}, we propose the following conjecture.

\begin{conjecture}
    Let $X$ be weakly pseudoconvex manifold, $\scr{J}$ an ideal sheaf on $X$, and $L\longrightarrow X$ be a holomorphic line bundle which is exhaustion singular-positive with $\scr{J}$,  
    i.e., the sequence of singular Hermitian metrics $\{h_j\}_{j\in\bb{N}}$ exists and satisfies the appropriate conditions.
    If $h_j\preccurlyeq h_{j+1}|_{X_{c_j}}$ for any $j\in\bb{N}$, does there exist a singular Hermitian metric $h$ on $L$, globally defined on $X$, which is singular positive on $X$ and satisfies $\scr{I}(h)=\scr{J}$ on $X$?
\end{conjecture}

If this conjecture is true, Remark-Theorem \ref{Remark-Theorem:} reduces to \cite[Theorem\,1.1]{Wat24}.

\section{Points separation and Embedding theorems}

In this section, let $X$ be a \textit{non}-\textit{compact} weakly pseudoconvex complex space of pure dimension $n$ and $L\longrightarrow X$ be a holomorphic line bundle. 
Let $\pi:\tl{X}\longrightarrow X$ be a canonical desingularization and $E:=\pi^{-1}(\xs)$ be the $\pi$-exceptional divisor.
In this setting, we present global points separation and global embedding theorems for $\tx\setminus E$ in the case where $L$ is positive.

\begin{theorem}[{\textnormal{\textbf{Global points separation}}}]\label{Theorem: global points separation}
    Let $x_1,\ldots,x_r$ be $r$ distinct points on an open subset $\tx\setminus E$.
    If $L$ is positive, then for any integer $m\geq\big(n(n+2r-3)\big)/2+2-r$,
    the space $H^0(\tx,K_{\tx}\otimes\pi^*L^{\otimes m})$ separates $\{x_j\}^r_{j=1}$, i.e., the following restriction map 
    \begin{align*}
        H^0(\tx,K_{\tx}\otimes\pi^*L^{\otimes m})\longrightarrow\bigoplus^r_{j=1}\cal{O}_{\tx}/\fra{m}_{X,x_j}
    \end{align*}
    is surjective. In particular, if $m\geq n(n-1)/2+1$ then the adjoint bundle $K_{\tx}\otimes\pi^*L^{\otimes m}$ is generated by global sections on $\tx\setminus E$.
\end{theorem}

Here, if $X$ is compact, then the condition on $m$ can be taken as $m\geq n(n+2r-1)/2+1$ (see \cite[Section\,8]{Wat24}).
For the proof of this theorem, we introduce the following.

\begin{lemma}\label{Lemma: exh sing posi tensor m}
    If $L$ is positive, then for any positive integer $m$, there exist a sequence of real numbers $\{c_j\}_{j \in \mathbb{N}}$ diverging to $+\infty$ and a sequence of singular Hermitian metrics $\{h_j\}_{j \in \mathbb{N}}$ of $\pi^*L$
    such that $\pi^*L^{\otimes m}$ becomes exhaustion singular-positive with $\cal{O}_{\tx}$ by means of the sequence $\{\tl{h}_j^m\}_{j\in\bb{N}}$, where each $\tl{h}_j$ is defined on $\pi^*L|_{\tl{X}_{c_j}}$.
    That is, for any $j\in\bb{N}$, the singular Hermitian metric $\tl{h}_j$ on $\pi^*L|_{\tl{X}_{c_j}}$ can be constructed such that $\tl{h}_j$ is singular positive on $\tl{X}_{c_j}$ and satisfies $\scr{I}(\tl{h}_j^m) = \cal{O}_{\tl{X}}$ on $\tl{X}_{c_j}$.
\end{lemma}

The desired singular Hermitian metrics $\tl{h}_j$ can be defined in the same manner as in the proof of Proposition \ref{Proposition: singular positivity on tlX_c by using SOP}, 
for some $\tau>0$, by taking a locally integrable function $\psi_j:\tl{X}_{c_j+\tau}\longrightarrow[-\infty,+\infty)$ and an appropriate constant $\varepsilon_{c_j}>0$ with $\varepsilon_{c_j}<2/\max_{x\in\overline{\tl{X}}_{c_j+\tau}}\nu(\psi_j,x)$, 
and setting $\tl{h}_j:=\tl{h}_{\varepsilon_{c_j}/m}=\pi^*h\cdot \exp(-\frac{\varepsilon_{c_j}}{m}\psi_j)$ on $\pi^*L|_{\tl{X}_{c_j}}$.

\begin{proof}[Proof of Theorem \ref{Theorem: global points separation}]
    Let $m\geq\big(n(n+2r-3)\big)/2+2-r$ be an integer, and take the sequence of real numbers $\{c_j\}_{j\in\bb{N}}$ diverging to $+\infty$ and the sequence of singular Hermitian metrics $\{\tl{h}_j\}_{j\in\bb{N}}$ from the above Lemma corresponding to this $m$; 
    that is, $\pi^*L^{\otimes m}$ is exhaustion singular-positive with $\cal{O}_{\tx}$ and $\scr{I}(\tl{h}_j^t)=\cal{O}_{\tx}$ on $\tl{X}_{c_j}$ for any $0<t\leq m$. Then there exists an integer $j_0 \in \mathbb{N}$ such that $\{x_1,\ldots,x_r\}\subset\tx_{c_{j_0}}$. 

    By applying \cite[Theorem\,7.5]{Wat24} to $(\pi^*L,\tl{h}_{j_0})$, the restriction map 
    \begin{align*}
        H^0(\tx_{c_{j_0}},K_{\tx}\otimes \pi^*L^{\otimes m})\longrightarrow\bigoplus^r_{j=1}\cal{O}_{\tx}/\fra{m}_{X,x_j}
    \end{align*}
    is surjective.
    In fact, the following inequalities for the symbols are known:
    \begin{align*}
        I_\Sigma(\pi^*L,\tl{h}_{j_0},\{x_j\}^r_{j=1})\leq m_0(\pi^*L,\tl{h}_{j_0},\{x_j\}^r_{j=1})\leq\frac{n(n+2r-3)}{2}+2-r,
    \end{align*}
    which are the symbols introduced in \cite[Section\,7]{Wat24}. 
    Since $\pi^*L^{\otimes m}$ is exhaustion singular-positive with the structure sheaf $\cal{O}_{\tx}$ by Theorem \ref{Theorem: exh sing positivity on tlX} or the above Lemma \ref{Lemma: exh sing posi tensor m}, 
    the proof is completed by applying the Runge-type approximation Theorem \ref{Theorem: ext approximation theorem}.
\end{proof}

\begin{theorem}[{\textnormal{\textbf{Global embedding}}}]\label{Theorem: global embedding}
    Let $m$ be a positive integer with $m\geq n(n+1)/2$. 
    If $L$ is positive, then the open subset $\tx\setminus E$ is holomorphically embeddable in to $\bb{P}^{2n+1}$ by a linear subsystem of $|(K_{\tx}\otimes \pi^*L^{\otimes m})^{\otimes n+1+\ell}\otimes \pi^*L^{\otimes p}|$ for any $\ell\geq1$ and $p\geq0$. 

    Furthermore, the adjoint bundle $K_{\tx}\otimes \pi^*L^{\otimes m}$ admits a singular Hermitian metric which is singular semi-positive on $\tx$, 
    and in particular, for any fixed relatively compact open subset of $\tl{X}\setminus E$, the metric can be constructed to be smooth on this subset.
\end{theorem}

\begin{proof}
    We take $\scr{X}$ as an arbitrary open subset of $\tx\setminus E$.
    By Theorem \ref{Theorem: global points separation}, for every relatively compact open subset $K$ of $\scr{X}$, 
    there exists a finite dimensional subspace $V_K\subset H^0(\tx,K_{\tx} \otimes \pi^*L^{\otimes m})$ such that the Kodaira map $\varPhi_{V_K}:\tx\dashrightarrow \bb{P}(V_K)$ is holomorphic and one to one over $K$, where $K\subset \tx\setminus\rom{Bs}_{|V_K|}$.

    Let $N:=\rom{dim}\,V_K-1$ and $\cal{L}:=K_{\tx}\otimes \pi^*L^{\otimes m}$, and let $h_{\cal{L}}$ be a smooth Hermitian metric on $\cal{L}$. For any point $x\in K$, there exists a singular Hermitian metric $\cal{H}_x$ on $\cal{L}^{\otimes n+1}$ over $\tx$ such that the curvature current is semi-positive, $x$ is isolated in $V(\scr{I}(\cal{H}_x))$, and that $\scr{I}(\cal{H}_x)_x\subset\fra{m}^2_{X,x}$ (see \cite[Lemma\,5.1]{Tak98}). 
    This metric \( \cal{H}_x \) can be constructed as follows. Let $(w_0:\cdots:w_N)$ be a homogenuous coordinate of $\bb{P}^N=\bb{P}(V_K)$ such that $\varPhi_{V_K}(x)=\{w_1=\cdots=w_N=0\}$. In other words, \( x \) is an isolated common zero of $\{\varPhi^*_{V_K}w_j\}^N_{j=1}$.
    For a local coordinate $(z_1,\ldots,z_n)$ centered at $x$, the inequality $\sum^N_{j=1}|\varPhi^*_{V_K}w_j|^2_{h_{\cal{L}}}\leq\cal{C}_x\sum^n_{j=1}|z_j|^2$ holds, where $\cal{C}_x>0$ is a constant. 
    Thus, the desired singular metric \( \cal{H}_x \) on \( \cal{L} \) is obtained by 
    \begin{align*}
        \cal{H}_x:=\frac{h_{\cal{L}}^{\otimes n+1}}{\big(\sum^N_{j=1}|\varPhi^*_{V_K}w_j|^2_{h_{\cal{L}}}\big)^{n+1}}=\frac{1}{\big(\sum^N_{j=1}|\varPhi^*_{V_K}w_j|^2\big)^{n+1}},
    \end{align*}
    which is defined on \( \tx\setminus\rom{Bs}_{|V_K|} \), is singular semi-positive, and satisfies the desired properties (see Skoda's result \cite[Lemma\,5.6\,(b)]{Dem12}).
    Locally, for some open neighborhood $U$ of each point in $\rom{Bs}_{|V_K|}$, the weight function $\varphi_{\cal{H}_x}:=(n+1)\log\sum^N_{j=1}|\varPhi^*_{V_K}w_j|^2$ of $\cal{H}_x$ is plurisubharmonic on $U\setminus\rom{Bs}_{|V_K|}$, and it tends to $-\infty$ as the point approaches $\rom{Bs}_{|V_K|}$. 
    Since $\varphi_{\cal{H}_x}$ is bounded above, it extends uniquely as a plurisubharmonic function. 
    Hence, by defining the value of $\cal{H}_x$ at $\rom{Bs}_{|V_K|}$ to be $+\infty$, the metric $\cal{H}_x$ can be redefined as a singular Hermitian metric of $\cal{L}$ over the whole $\tx$, and it is singular semi-positive on $\tx$.

    Let $h_{V_K}$ be a natural singular Hermitian metric on $\cal{L}$ defined by $V_K$, then $h_{V_K}$ has weight function $\varphi_{V_K}=\log\sum^N_{j=0}|\varPhi^*_{V_K}w_j|^2$ and is singular semi-positive on $\tx$. In particular, we have $\scr{I}(\cal{H}_x)\subseteq\scr{I}(h^{n+1}_{V_K})$. 
    Fix $\ell\geq1$ and $p\geq0$. Here, the singular Hermitian metric $h_{V_K}^{n+\ell}$ on $\cal{L}^{\otimes n+\ell}$ is singular semi-positive on $\tl{X}$.
    From Theorem \ref{Theorem: exh sing positivity on tlX twisted by psef}, similarly to Lemma \ref{Lemma: exh sing posi tensor m}, there exist a sequence of real numbers $\{c_j\}_{j\in\bb{N}}$ diverging to $+\infty$ and a sequence of singular Hermitian metrics $\{\tl{h}_j\}_{j\in\bb{N}}$ defining $\pi^*L$ to be exhaustion singular-positive with $\cal{O}_{\tx}$ 
    such that the sequence $\{\tl{h}_j^{p+1}\}_{j\in\bb{N}}$ defines $\pi^*L^{\otimes p+1}$ to be exhaustion singular-positive with $\cal{O}_{\tx}$ and the sequence $\{h_{V_K}^{n+\ell}\otimes \tl{h}_j^{p+1}\}_{j\in\bb{N}}$ defines $\cal{L}^{\otimes n+\ell}\otimes\pi^*L^{\otimes p+1}$ to be exhaustion singular-positive with $\scr{I}(h_{V_K}^{n+\ell})$. 

    Here, there exists an integer $j_0\in\bb{N}$ such that $K\subset \tx_{c_{j_0}}\setminus E$. For simplicity, let $c := c_{j_0}$.
    Let $\hbar_x:=\cal{H}_x\otimes h_{V_K}^{\ell-1}\otimes \tl{h}_{j_0}^{p+1}$, then the singular Hermitian metric $\hbar_x$ on $\cal{L}^{\otimes n+\ell}\otimes\pi^*L^{\otimes p+1}|_{\tl{X}_c}$ is singular positive and satisfies $\scr{I}(\hbar_x)=\scr{I}(\cal{H}_x\otimes h_{V_K}^{\ell-1})\subset \scr{I}(h^{n+\ell}_{V_K})=\scr{I}(h^{n+\ell}_{V_K}\otimes\tl{h}_j^{p+1})$ on $\tl{X}_{c}$. 
    Since $h_{V_K}$ and $\tl{h}_j$ are smooth on $K$, we obtain that $\scr{I}(\hbar_x)_x=\scr{I}(\cal{H}_x)_x\subset\fra{m}^2_{X,x}$ and $x$ is isolated in $V(\scr{I}(\hbar_x))$. 
    Let $\scr{J}$ be the ideal sheaf on $\tl{X}_c$ which agrees with $\scr{I}(\hbar_x)$ on $\tl{X}_c\setminus\{x\}$ and which agrees with $\cal{O}_{\tx}$ at $x$. 
    From the exact sequence $0\longrightarrow\scr{I}(\hbar_x)\longrightarrow\scr{J}\longrightarrow\scr{J}/\scr{I}(\hbar_x)\longrightarrow0$ and the surjective $\scr{J}/\scr{I}(\hbar_x)\twoheadrightarrow\cal{O}_{\tx}/\fra{m}^2_{\tx,x}$, 
    applying the Nadel vanishing Theorem (see \cite[Theorem 6.2]{Wat24}) to $(\cal{L}^{\otimes n+\ell}\otimes \pi^*L^{\otimes p+1},\hbar_x)$ on $\tl{X}_c$, for any $x\in K$ the following restriction map is surjective.
    \begin{align*}
        H^0(\tl{X}_c,K_{\tx}\otimes\cal{L}^{\otimes n+\ell}\otimes \pi^*L^{\otimes p+1}\otimes\scr{J})\longrightarrow\cal{O}_X/\fra{m}^2_{X,x}.
    \end{align*}
    Recalling that $\scr{I}(\hbar_x)\subset\scr{I}(h^{n+\ell}_{V_K})$ on $\tl{X}_c$ and $\scr{I}(h^{n+\ell}_{V_K})_x=\cal{O}_{\tx,x}$, the definition of the ideal $\scr{J}$ yields $\scr{J}\subset\scr{I}(h^{n+\ell}_{V_K})$ on $\tx_c$, and the following restriction map is surjective. 
    \begin{align*}
        H^0(\tl{X}_c,K_{\tx}\otimes\cal{L}^{\otimes n+\ell}\otimes \pi^*L^{\otimes p+1}\otimes\scr{I}(h^{n+\ell}_{V_K}))\longrightarrow\cal{O}_X/\fra{m}^2_{X,x}.
    \end{align*}

    By \cite[Lemma\,5.1]{Tak98} and the above extension to a natural singular Hermitian metric, for any distinct two points $x$ and $y$ on $K$, there exists a singular Hermitian metric $\cal{H}_{x,y}$ on $\cal{L}^{\otimes n+1}$ 
    such that the curvature current is semi-positive, $\{x,y\}\subset V(\scr{I}(\cal{H}_{x,y}))$ and $x$ is isolated in $V(\scr{I}(\cal{H}_{x,y}))$. 
    Let \( \hbar_{x,y}:=\cal{H}_{x,y}\otimes h_{V_K}^{\ell-1}\otimes \tl{h}_{j_0}^{p+1} \), then the singular Hermitian metric $\hbar_{x,y}$ on $\cal{L}^{\otimes n+\ell}\otimes \pi^*L^{\otimes p+1}|_{\tl{X}_c}$ is singular positive and satisfies $\scr{I}(\hbar_{x,y})=\scr{I}(\cal{H}_x\otimes h_{V_K}^{\ell-1})\subset\scr{I}(h_{V_K}^{n+\ell})$ on $\tl{X}_c$. 
    Similarly applying the Nadel vanishing Theorem (see \cite[Theorem 6.2]{Wat24}), there exists a holomorphic section \( \sigma\in H^0(\tl{X}_c,K_{\tx}\otimes\cal{L}^{\otimes n+\ell}\otimes \pi^*L^{\otimes p+1}\otimes\scr{I}(h_{V_K}^{n+\ell})) \) that satisfies \( \sigma(x)\ne\sigma(y) \).
    
    Since $\cal{L}^{\otimes n+\ell}\otimes\pi^*L^{\otimes p+1}$ is exhaustion singular-positive with $\scr{I}(h_{V_K}^{n+\ell})$, by applying the refined Runge-type approximation Theorem \ref{Theorem: ext approximation theorem}, each section obtained on \( X_c \) can be obtained as a section on the entire \( X \).
    Hence, every relatively open subset $K$ of $\scr{X}$ can be embedded into a projective space by some finite dimensional subspace of $H^0(\tx,(K_{\tx}\otimes \pi^*L^{\otimes m})^{\otimes n+1+\ell}\otimes \pi^*L^{\otimes p})$. 
    Here, $\scr{X}$ is countable at infinity.
    Then this theorem follows from well-known argument of Whitney (cf. \cite[Chapter\,V]{Hor90}). 
\end{proof}

Theorem \ref{Theorem: bigness theorem and Vol>0} follows from Theorems \ref{Theorem: global points separation}, \ref{Theorem: global embedding}, and \cite[Theorem\,8.6]{Wat24}. 
By the refined Runge-type approximation Theorem \ref{Theorem: ext approximation theorem}, in the case where no singularities are present, we argue as in Takayama's theorem \cite[Theorem\,1.2]{Tak98} and the proof above to obtain the following embedding theorem assuming only the existence of an exhaustion positive line bundle, 
and furthermore we obtain the following corollary by using Example \ref{Example: Tak98b Theorem}, that is, \cite[Theorem\,1.3]{Tak98b}.

\begin{theorem}\label{Theorem: embedding only exhaustion positive}
    Let $X$ be a non-compact weakly pseudoconvex manifold of dimension $n$ and $L\longrightarrow X$ be a holomorphic line bundle. 
    If $L$ is exhaustion positive, then for any $m\geq n(n+1)/2$, the adjoint bundle $K_X\otimes L^{\otimes m}$ is ample on $X$ and the manifold $X$ can be holomorphically embedded into $\bb{P}^{2n+1}$ by a linear subsystem of $|(K_X\otimes L^{\otimes m})^{\otimes (n+2)}|$. 
\end{theorem}

\begin{corollary}\label{Corollary: embedding of toroidal group}
    Let $T$ be a toroidal group of dimension $n$ and $L\longrightarrow T$ be a holomorphic line bundle with a \kah form in the first Chern class $c_1(L)$. 
    Then $T$ can be holomorphically embedded into $\bb{P}^{2n+1}$ by a linear subsystem of $|(K_T\otimes L^{\otimes m})^{\otimes (n+2)}|$ for any $m\geq n(n+1)/2$. 
    In particular, the adjoint bundle $K_T\otimes L^{\otimes m}$ is ample on $T$.
\end{corollary}

Finally, we state the following conjecture.

\begin{conjecture}
    Let $X$ be a non-compact weakly pseudoconvex complex space of pure dimension $n$, $\pi:\tx\longrightarrow X$ be a canonical desingularization and $L\longrightarrow X$ be a holomorphic line bundle. 
    If $L$ is positive, then 
    \begin{itemize}
        \item [($a$)] does there exist an integer $m \geq n(n+1)/2$ such that the adjoint bundle $K_{\tx} \otimes \pi^*L^{\otimes m}$ of the pullback $\pi^*L$ is semi-ample? 
        \item [($b$)] does there exist an integer $m \geq n(n+1)/2$ such that the adjoint bundle $\ogr\otimes L^{\otimes m}$ is ample? 
                    Furthermore, is the tensor power of the adjoint bundle $(\ogr\otimes L^{\otimes m})^{\otimes (n+2)}$ very ample? 
                    In other words, does the linear system $|(\ogr\otimes L^{\otimes m})^{\otimes (n+2)}|$ give a holomorphic embedding of $X$ into $\mathbb{P}^{2n+1}$?
    \end{itemize}

    Here, the sheaf $\ogr$ denotes the Grauert-Riemenschneider canonical sheaf on $X$, which satisfies $\ogr=\pi_*\omega_{\tx}$.
\end{conjecture}

Regarding $(a)$, since $\pi^{*}L\longrightarrow\tx$ is semi-positive and exhaustion singular-positive with $\cal{O}_{\tx}$, the strength of positivity on each sublevel $\tl{X}_c$ corresponds to something comparable to being nef and big. 
Hence, the base point free theorem on compact manifolds suggests that the statement should hold, and a non-compact extension of the theorem is required.

Regarding $(b)$, on each sublevel $X_c$, the ampleness of $L^{\otimes p_c}$ is known for sufficiently large integers $p_c$ (see \cite{Fuj75}). 
Since the adjoint bundle $\ogr\otimes L^{\otimes m_c}$ becomes positive on $X_c$ for an large integer $m_c$, a similar argument shows that $(\ogr\otimes L^{\otimes m_c})^{\otimes q_c}$ is ample on $X_c$ for sufficiently large integers $q_c$. 
This gives a partial confirmation of the statement. If these integers $q_c$ can be bounded independently of $c$, and if an approximation theorem for holomorphic sections can be established in the same manner on $X$ with singularities, then the expectation is likely to be settled.


\section{Union Problem}

\begin{proof}[Proof of Theorem \ref{Theorem: Union Problem on w.p.c}]
    To avoid confusion, we rewrite the set $\Omega$ as $X:=\Omega$.
    For the smooth exhaustion plurisubharmonic function $\varPsi:X\longrightarrow[-\infty,+\infty)$ that exists by weakly pseudoconvexity, set its sublevel sets as $X_\nu:=\{x\in X\mid\varPsi(x)<\nu\}$.
    Since $X_\nu$ is relatively compact, the compact set $\overline{X_\nu}$ can be covered by finitely many open sets from the covering $\{\Omega_j\}_{j\in\bb{N}}$.
    As the covering is increasing, by taking the largest index $j_\nu$ among them, we obtains $\overline{X_\nu} \subset \Omega_{j_\nu}$.

    Therefore, each sublevel set $X_\nu$ is also Stein. 
    In fact, using a smooth strictly plurisubharmonic function $\psi_{j_\nu}$ on $\Omega_{j_\nu}$, the smooth function $\tl{\varPsi}_\nu:= \frac{1}{\nu-\varPsi} + \psi_{j_\nu}$ on $X_\nu$ is exhaustion and strictly plurisubharmonic.
    Hence, the sequence of Stein domains $\{X_\nu\}_{\nu\in\bb{N}}$ is also exhausion of $X=\Omega$. 

    The anti-canonical line bundle $K^{*}_X$ is exhaustion positive (see Example \ref{Example: exhaustion positivity}). 
    In fact, for a smooth Hermitian metric $h^*_K$ on $K^{*}_X$, there exists a positive number $m_\nu>0$ such that a smooth Hermitian metric $h^*_K\cdot e^{-m_\nu\psi_{j_\nu}}$ on $K^*_X|_{X_\nu}$ is positive on $X_\nu$ by relative compact-ness of $X_\nu$.
    By the Runge-type approximation Theorem \ref{Theorem: ext approximation theorem} (or Proposition \ref{Proposition: approximation thm with L2-norms and semi-norms} or \cite{Ohs82,Ohs83}), for any $\nu\in\bb{N}$, the restriction map 
    \begin{align*}
        \rho_\nu:H^0(X_{\nu+1},\cal{O}_X)\longrightarrow H^0(X_\nu,\cal{O}_X)
    \end{align*}
    has dense image with respect to the topology of uniform convergence on all compact subsets on $X_\nu$; in other words, the pair $(X_\nu,X_{\nu+1})$ is Runge. 
    Therefore, by Stein's result, i.e., Theorem \ref{Theorem: Stein's result}, the union set $\Omega=:X=\bigcup_{\nu\in\bb{N}}X_\nu$ is Stein.
\end{proof}

In the setting of this Theorem \ref{Theorem: Union Problem on w.p.c}, the anti-canonical bundle $K^*_X$ is exhaustion positive, and from the result of Markoe and Silve (see \cite{Mar77,Sil78}) it follows that $H^1(X,\mathcal{O}_X)=0$. 
In view of this, we propose the following conjecture: 

\begin{conjecture}
    Let $X$ be a weakly pseudoconvex manifold and $L\longrightarrow X$ be a holomorphic line bundle. 
    If $L$ is exhaustion singular-positive with an ideal sheaf $\scr{J}$ on $X$, does the following cohomology vanish?
    \begin{align*}
        H^q(X,K_X\otimes L\otimes\scr{J})=0,
    \end{align*}
    for any $q>0$. 
\end{conjecture}


\vspace*{3mm}
\noindent
{\bf Acknowledgement.} 
The author is supported by Grant-in-Aid for Research Activity Start-up $\sharp$24K22837 from the Japan Society for the Promotion of Science (JSPS).



\end{document}